\documentclass[journal]{IEEEtran}
\usepackage{wrapfig}
\usepackage{color}
\usepackage[T1]{fontenc}
\usepackage{fancyhdr}
\usepackage{colortbl}
\usepackage{comment}
\usepackage{cite}
\usepackage{changepage}

\ifCLASSINFOpdf
\usepackage[pdftex]{graphicx}
\else
\fi

\usepackage{hhline}

\usepackage{enumerate}

\usepackage{amsmath, amssymb,latexsym, amsthm}

\theoremstyle{remark}
\newtheorem{remark}{Remark}

\usepackage{xargs}                      
\usepackage[pdftex,dvipsnames]{xcolor} 
\usepackage[colorinlistoftodos,prependcaption,textsize=tiny]{todonotes}
\newcommand{\Comments}{1} 
\newcommandx{\unsure}[2][1=]{\todo[linecolor=red,backgroundcolor=red!25,bordercolor=red,#1]{#2}}
\newcommandx{\change}[2][1=]{\todo[linecolor=blue,backgroundcolor=blue!25,bordercolor=blue,#1]{#2}}
\newcommandx{\info}[2][1=]{\todo[linecolor=OliveGreen,backgroundcolor=OliveGreen!25,bordercolor=OliveGreen,#1]{#2}}
\newcommandx{\improvement}[2][1=]{\todo[linecolor=Plum,backgroundcolor=Plum!25,bordercolor=Plum,#1]{#2}}
\newcommandx{\thiswillnotshow}[2][1=]{\todo[disable,#1]{#2}}
\ifnum\Comments=1               
\fi

\newcommand{\kyri}[1]{\textcolor{blue}}
\newcommand{\kg}[1]{\textcolor{green}}

\usepackage{algorithmic}
\usepackage{epsfig, amsmath, color}

\usepackage{array}

\usepackage{siunitx}
\usepackage{accents}

\usepackage{color}

\newcommand{\probl}{\mathcal{P}_1}
\newcommand{\btx}{\tilde{\mathbf{x}}}
\newcommand{\bhx}{\hat{\mathbf{x}}}

\newcommand{\Pg}{{P}_{\rm grid}}
\newcommand{\Pch}{{P}_{\rm ch}}
\newcommand{\Pd}{{P}_{\rm dis}}

\newcommand{\argmin}{\text{argmin}}
\newcommand{\ts}{t^*}

\def\R{\mathbb{R}}
\newcommand{\st}[1]{#1^{(t)}}

\newcommand{\Emax}{E_{\rm max}}
\newcommand{\Emin}{E_{\rm min}}
\newcommand{\etac}{\eta_{c}}
\newcommand{\etad}{\eta_{d}}
\newcommand{\fcost}{f_{\rm cost}}
\newcommand{\tPch}{\tilde{P}_{\rm ch}}
\newcommand{\tPdis}{\tilde{P}_{\rm dis}}
\newcommand{\hPch}{\hat{P}_{\rm ch}}
\newcommand{\hPdis}{\hat{P}_{\rm dis}}
\newcommand{\Pmaxc}{P_{\rm ch,max}}
\newcommand{\Pmaxd}{P_{\rm dis,max}}
\newcommand{\tPgrid}{\tilde{P}_{\rm grid}}
\newcommand{\hPgrid}{\hat{P}_{\rm grid}}
\newcommand{\Psol}{P_{\rm sol}}
\newcommand{\Pc}{P_{\rm c}}
\newcommand{\tPc}{\tilde{P}_{\rm c}}
\newcommand{\hPc}{\hat{P}_{\rm c}}
\newcommand{\xvec}{\underline{\mathbf{x}}}
\newcommand{\xopt}{\underline{\mathbf{\tilde{x}}}}
\newcommand{\xhat}{\underline{\mathbf{\hat{x}}}}

\newcommand{\fgridl}{\underaccent{\bar}{f}_{\rm grid}}
\newcommand{\fsoll}{\underaccent{\bar}{f}_{\rm sol}}
\newcommand{\fsolu}{\bar{f}_{\rm sol}}
\newcommand{\fchl}{\underaccent{\bar}{f}_{\rm ch}}
\newcommand{\fchu}{\bar{f}_{\rm ch}}
\newcommand{\fdisl}{\underaccent{\bar}{f}_{\rm dis}}
\newcommand{\fdisu}{\bar{f}_{\rm dis}}
\newcommand{\fsocl}{\underaccent{\bar}{f}_{\rm soc}}
\newcommand{\fsocu}{\bar{f}_{\rm soc}}

\newcommand{\lgridl}{\underaccent{\bar}{\lambda}_{\rm grid}}
\newcommand{\lsoll}{\underaccent{\bar}{\lambda}_{\rm sol}}
\newcommand{\lchl}{\underaccent{\bar}{\lambda}_{\rm ch}}
\newcommand{\ldisl}{\underaccent{\bar}{\lambda}_{\rm dis}}
\newcommand{\lsocl}{\underaccent{\bar}{\lambda}_{\rm soc}}

\newcommand{\lsolu}{\bar{\lambda}_{\rm sol}}
\newcommand{\lchu}{\bar{\lambda}_{\rm ch}}
\newcommand{\ldisu}{\bar{\lambda}_{\rm dis}}
\newcommand{\lsocu}{\bar{\lambda}_{\rm soc}}

\newtheorem{proposition}{Proposition}

\usepackage[normalem]{ulem}

\begin{document}
	\title{Control of Energy Storage in Home Energy Management Systems: Non-Simultaneous Charging and Discharging Guarantees}
	
	\author{Kaitlyn~Garifi,~\IEEEmembership{Student Member,~IEEE,}
		Kyri~Baker,~\IEEEmembership{Member,~IEEE,}
		Dane~Christensen,~\IEEEmembership{Member,~IEEE,}
		and~Behrouz~Touri,~\IEEEmembership{Member,~IEEE}
		\thanks{K. Garifi and K. Baker are with the Electrical, Computer, and Energy Engineering Department, University of Colorado Boulder, Boulder,
			CO, 80309 USA (e-mail: \{kaitlyn.garifi;kyri.baker\}@colorado.edu).}
		\thanks{D. Christensen is with the National Renewable Energy Laboratory, Golden, CO, 80401 USA (email: dane.christensen@nrel.gov).}%
		\thanks{B. Touri is with the Electrical and Computer Engineering Department, University of California San Diego, La Jolla, CA, 92093 USA (email: btouri@ucsd.edu).}
		\thanks{Original manuscript received 30 June 2018. This work was supported in parts by the Air Force Office of Scientific Research under the AFOSR-YIP award FA9550-16-1-0400 and the US Department of Energy Building Technologies Office under contract DE-AC36-08GO28308.}}

	\maketitle
	
	\begin{abstract}
		In this paper we provide non-simultaneous charging and discharging guarantees for a linear energy storage system (ESS) model for a model predictive control (MPC) based home energy management system (HEMS) algorithm. The HEMS optimally controls the residential load and residentially-owned power sources, such as photovoltaic (PV) power generation and energy storage, given residential customer preferences such as energy cost sensitivity and ESS lifetime. Under certain problem formulations with a linear ESS model, simultaneous charging and discharging can be observed as the optimal solution when there is high penetration of PV power. We present analysis for a proposed HEMS optimization formulation that ensures non-simultaneous ESS charging and discharging operation for a linear ESS model that captures both charging and discharging efficiency of the ESS. The energy storage system model behavior guarantees are shown for various electricity pricing schemes such as time of use (TOU) pricing and net metering. Simulation results demonstrating desirable ESS behavior are provided for each electricity pricing scheme.	\end{abstract}
	\vspace{-0.1cm}
	\begin{IEEEkeywords}
		model predictive control, optimization, energy storage systems, home energy management systems
	\end{IEEEkeywords}	
	\vspace{-0.5cm}
	\IEEEpeerreviewmaketitle
	\section*{Nomenclature}
	\vspace{-0.4cm}
	\begin{table}[h!]
		\centering 
		\begin{tabular}{lll}
			$\alpha$ & Penalty coefficient for ESS charging \\
			$\beta$ & Penalty coefficient for ESS discharging \\
			$c_{e}$		&  Cost of $P_{\rm grid}$ (\$/kWh)\\
			$E$ & ESS state of charge (kWh)\\
			$E_{\rm max}$ & Maximum energy storage in ESS (kWh)\\
			$E_{\rm min}$ & Minimum energy storage in ESS (kWh)\\
			$\eta_c$ & ESS charging efficiency \\
			$\eta_d$ & ESS discharging efficiency \\
			$\Pc$ & Curtailed solar power (kW)\\
			$P_{\rm ch}$ & Power injected into ESS (kW) \\
			$P_{\rm dis}$ & Power drawn from ESS (kW)\\
			$P_{\rm grid}$	&  Power consumed from the grid (kW)\\
			$P_{\rm L,house}$	&  Total residential load (kW)\\
			$\Pmaxc$ & Maximum charging power (kW)\\
			$\Pmaxd$ & Maximum discharging power (kW)\\		
			$P_{\rm sol}$	&  Available solar power (kW)
		\end{tabular}
	\vspace{-0.3cm}
	\end{table}
	\vspace{-0.5cm}
	\section{Introduction} \vspace{-0.2cm}
	\IEEEPARstart{I}{ntegrating} renewable energy into the power grid has led to both generation-side and demand-side solutions to address the intermittent nature of these renewable energy resources. Electric energy storage systems (ESS) are commonly used to cope with the variability in renewable energy resources.~In particular, demand-side residential energy management solutions can be used to address stable renewable energy integration since residential buildings account for over 37.6\% of total electricity consumption in the U.S. \cite{buildingdata}. Home energy management systems (HEMS) provide residential demand-side energy management by coordinating residentially-owned power sources, appliances, user preferences, and renewable energy forecasts \cite{jin2017foresee,wu2015stochastic,zhou2016smart,HUSSAIN2015,Garifi18pes}. Many HEMS systems are equipped with an ESS due to decreasing battery costs and gained flexibility in responding to demand response (DR), and energy cost savings for the customer in various electricity pricing schemes such as time-of-use (TOU) pricing \cite{batteryecon,erdinc2015,Ghazvini2017,KHAN2016,HUSSAIN2015,Garifi18pes}. However, when incorporating ESS models into HEMS, ensuring proper ESS dynamics can be limiting due to the use of lossless or non-convex ESS operation models, or the use of restrictive computation methods. In this work, we provide non-simultaneous battery charging and discharging operation guarantees for a widely-used ESS model that captures both charging and discharging efficiency for use in a model predictive control (MPC) based HEMS algorithm that coordinates power drawn from the grid, available solar power, an ESS, and total residential load.
	
	To address this, we use the Karush-Kuhn-Tucker (KKT) optimality conditions to show that solutions to the convex MPC-based HEMS algorithm where the ESS is simultaneously charging and discharging are suboptimal under certain conditions. Similar analysis on ESS charging and discharging dynamics derived from the KKT conditions has been applied to a multi-period optimal power flow (OPF) problem with ESS assets \cite{kkt_opf2013}. For the OPF problem, the authors in \cite{kkt_opf2013} use the natural relation between locational marginal prices (LMPs) and the KKT conditions, to determine conditions for non-simultaneous ESS charging and discharging dynamics. We apply similar analysis to give guarantees on proper ESS dynamics in an MPC-based HEMS optimization framework for certain electricity pricing schedules and in the presence or absence of net metering. The penalty approach for discouraging simultaneous charging and discharging in the same ESS model is used in \cite{Zarrilli2018}; however, we provide a formal proof for the conditions under which this approach works. Additionally, the authors in~\cite{7792609} provide analysis on simultaneous charging and discharging using the same ESS model for a distributed power system with multiple grid-connected storage systems. 
	
	
	The outline of this paper is as follows: in Section~\ref{sec:relatedwork}, we survey various ESS models used in HEMS literature. In Section~\ref{sec:problemform}, the mathematical HEMS models and the overall MPC-based HEMS optimization problem are introduced. In Section~\ref{sec:proofs}, we present the main theoretical results of this paper showing that simultaneous ESS charging and discharging is suboptimal for various scenarios. In Section~\ref{sec:sims}, we provide simulation results highlighting proper ESS behavior for various electricity pricing schemes and net metering considered in this paper. We also provide simulation results showing when simultaneous ESS charging and discharging is an optimal solution. In Section~\ref{sec:conclusion}, conclusions regarding the use of the proposed ESS model are discussed, as well as potential areas of future work. \vspace{-0.25cm}
	\section{Commonly Used ESS Models}\label{sec:relatedwork}
	Energy storage systems are often included in renewable energy research due to their energy management flexibility \cite{hemmati2017technical}. Additionally, an ESS can be used to account for uncertainties in renewable energy such as solar and wind energy \cite{baker2017optEss,hemmati2017technical}. In HEMS research, many homes are equipped with an ESS due to its flexibility to participate in grid services such as demand response \cite{jin2017foresee,Garifi18pes,HUSSAIN2015} or help with consumer energy cost savings in real-time or demand pricing energy markets \cite{wu2015stochastic,batteryecon,zhou2016smart}. However, ensuring proper ESS dynamics (i.e. ensuring the model does not allow a physically unrealizable optimal control policy where the ESS simultaneously charges and discharges in the same time step) when incorporating ESS models into the HEMS framework leads to the use of lossless, binary, or non-convex models. 
	
	To ensure non-simultaneous ESS charging and discharging, a lossless ESS model is often used. A common lossless model used in HEMS literature is:
	\begin{subequations}
		\begin{align}
		E^{(t+1)}&=E^{(t)}+\Delta t P_{\rm ess}^{(t)}, \label{eqn:batt_lossless_begin}\\
		E_{\rm min}&\leq E^{(t+1)} \leq E_{\rm max},\\
		-P_{\rm max}& \leq P_{\rm ess}^{(t)} \leq P_{\rm max}, \label{eqn:batt_lossless_end}
		\end{align}
	\end{subequations}
	for all time $t\in\{1,\dots,N\}$ where $P_{\rm ess}^{(t)}$ is the power injected or drawn into the ESS \cite{yu2017online,CDC2016}. The model in~\eqref{eqn:batt_lossless_begin}-\eqref{eqn:batt_lossless_end} ensures non-simultaneous ESS charging and discharging since the power injected into the ESS and drawn from the ESS are captured in one variable $P_{\rm ess}^{(t)}$ which represents discharging when negative and charging when positive. While this model is linear, it assumes perfect ESS charging and discharging efficiency, implying the roundtrip efficiency of the ESS is 100\% \cite{yu2017online,CASTILLO2014885}, which does not accurately model the non-negligible losses of the system.
	
	Additionally, non-convex ESS models are used in HEMS research to capture ESS charging and discharging efficiency. A common non-convex ESS model that includes both charging and discharging terms and ensures the ESS does not simultaneously charge and discharge uses binary variables and is given by \cite{hemmati2017technical,chen2013mpc,2015robustOpt_wBatt_wRERs,wu2015stochastic,erdinc2015,Ghazvini2017,ParisioEssMILP}:
	\begin{subequations}
		\begin{align}
		E^{(t+1)}&=E^{(t)}+\Delta t(\eta_cP_{\rm ch}^{(t)}b^{(t)}-\eta_dP_{\rm dis}^{(t)}(1-b^{(t)})),\\
		E_{\rm min}&\leq E^{(t+1)} \leq E_{\rm max},\\
		0&\leq P_{\rm ch}^{(t)}\leq P_{\rm ch,max},\\
		0&\leq P_{\rm dis}^{(t)} \leq \Pmaxd,\\
		b^{(t)}&\in\{0,1\},
		\end{align}
	\end{subequations}
	for all time $t\in\{1,\dots,N\}$. The use of separate terms for ESS charging and discharging, $P_{\rm ch}^{(t)}$ and $P_{\rm dis}^{(t)}$, respectively, allow for a roundtrip efficiency of less than 100\% which accounts for losses in ESS-to-grid interactions \cite{CASTILLO2014885,2015robustOpt_wBatt_wRERs}. However, the use of binary variables in ESS models results in a non-convex model requiring the use of computationally restrictive methods such as mixed integer (non-)linear programming (MILP/MINLP) \cite{hemmati2017technical,chen2013mpc,2015robustOpt_wBatt_wRERs,wu2015stochastic}. Alternatively, the following nonlinear, non-convex ESS model without binary variables can be used:
	\begin{subequations}\label{eqn:essmdl_comp}
		\begin{align}
		E^{(t+1)}&=E^{(t)}+\Delta t\left(\eta_cP_{\rm ch}^{(t)}-\eta_dP_{\rm dis}^{(t)}\right), \label{eqn:ess_pchpdis0_begin}\\
		E_{\rm min}&\leq E^{(t+1)} \leq E_{\rm max},\\
		0&\leq P_{\rm ch}^{(t)}\leq P_{\rm ch,max},,\\
		0&\leq P_{\rm dis}^{(t)} \leq \Pmaxd,\\
		P_{\rm ch}^{(t)}&\cdot P_{\rm dis}^{(t)}=0, \label{eqn:pchpdiseq0}
		\end{align}
	\end{subequations} 
for all time $t\in\{1,\dots,N\}$ where the non-convex constraint in~\eqref{eqn:pchpdiseq0} is included in the model to ensure non-simultaneous ESS charging and discharging. Similar to the ESS model with binary variables, the model in~\eqref{eqn:ess_pchpdis0_begin}-\eqref{eqn:pchpdiseq0} also requires the use of computationally restrictive non-convex optimization solvers.
	
	Thus, the common models used to ensure non-simultaneous ESS charging and discharging either assume perfect 100\% roundtrip efficiency or are non-convex requiring the use of computationally limiting numerical methods. In this work, we provide non-simultaneous ESS charging and discharging guarantees for a frequently used linear ESS model that captures both charging and discharging efficiency. The main contributions of this work are the following:
	\begin{itemize}
		\item We present a formalized and thorough analysis that shows simultaneous charging and discharging is \textit{suboptimal} for a frequently used linear ESS model that captures both charging and discharging efficiency under the proposed HEMS optimization formulation, including under various pricing schemes and net metering. We believe this to be the first manuscript that formally proves these claims.
		\item The use of non-convex or mixed integer ESS models to ensure non-simultaneous charging and discharging are \textit{unnecessary} under the proposed HEMS optimization formulation. 
		\item We present situations under which simultaneous charging and discharging \textit{is} optimal in the HEMS optimization, and provide a means to ensure an equivalent solution where simultaneous charging and discharging does \textit{not} occur without adopting a non-convex ESS model.
	\end{itemize}
	\begin{figure}[t!]
		\centering
		\includegraphics[width=2.5in]{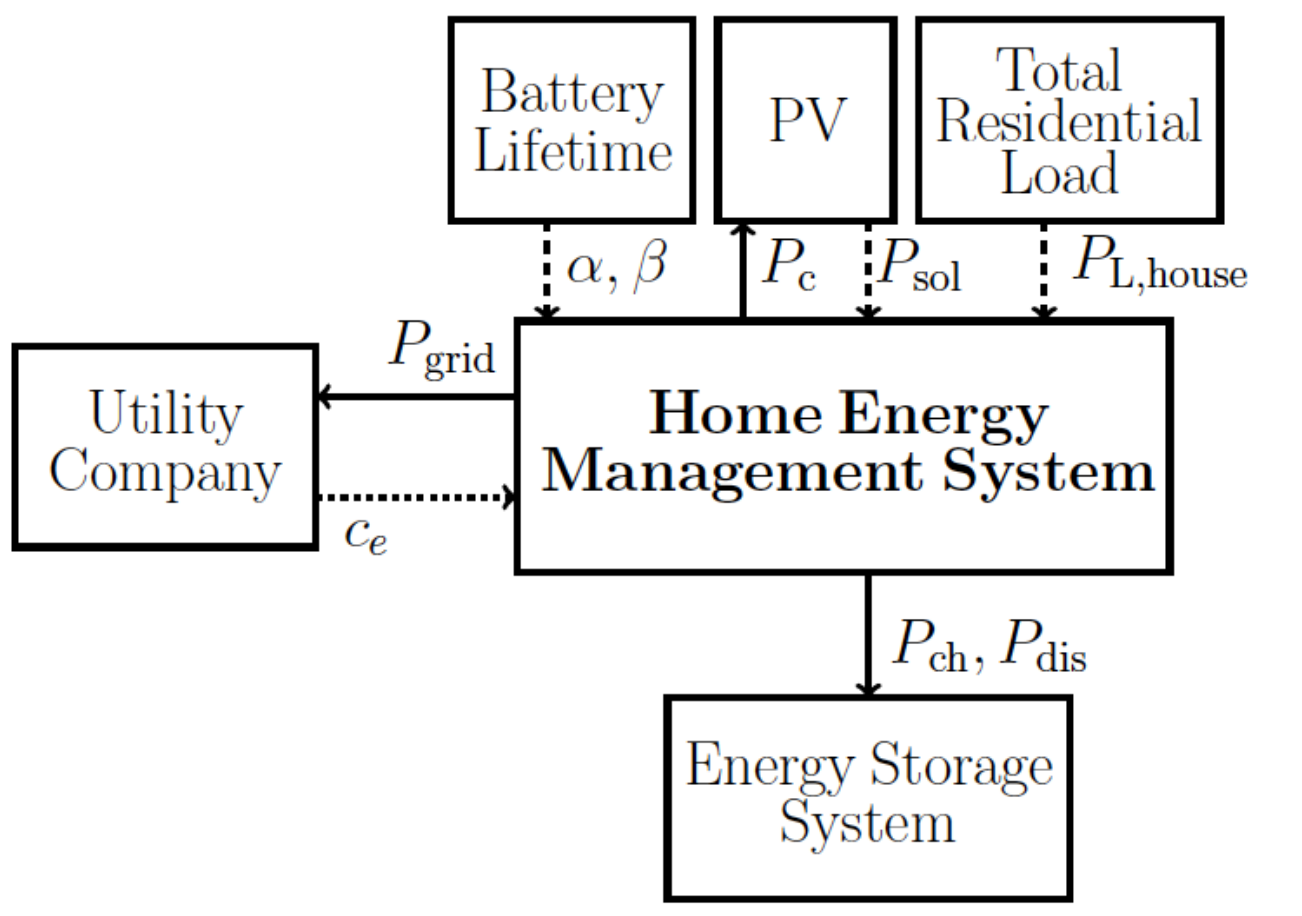} \vspace{-0.3cm}
		\caption{Overall HEMS control schematic.} \vspace{-0.3cm}
		\label{fig:hems_fig}
	\end{figure} 

	\section{Problem Formulation} \label{sec:problemform} \vspace{-0.1cm}
	In this section, we provide the mathematical models for the HEMS and the overall MPC optimization problem. In this work, we assume that the HEMS must coordinate the residential PV generation, the ESS, the total residential load, and power drawn from the grid needed to satisfy the load, as shown in Fig.~\ref{fig:hems_fig}.
	
	The residential energy storage system state of charge (SOC) and power charged/discharged are modeled by:
	\begin{subequations}
		\begin{align}
		E^{(t+1)}&=E^{(t)} + \eta_c \Delta tP_{\rm ch}^{(t)} -\eta_d\Delta tP_{\rm dis}^{(t)}, \label{eqn:batt_begin}\\
		E_{\rm min} &\leq E^{(t+1)} \leq E_{\rm max}, \\ 
		0&\leq P_{\rm ch}^{(t)}\leq P_{\rm ch,max},\\
		0&\leq P_{\rm dis}^{(t)} \leq \Pmaxd,\\
		0&<\eta_c<1<\eta_d, \label{eqn:batt_end}
		\end{align}
	\end{subequations}
	for all $t\in\{1,\dots,N\}$ where $E_{\rm min}$, $E_{\rm max}$ define the SOC limits. The charging efficiency is denoted $\eta_c$ and the discharging efficiency is denoted $\eta_{d}$. The condition in~\eqref{eqn:batt_end} can be relaxed and the results we present in this paper hold for any $\eta_c$ and $\eta_d$ satisfying $0<\eta_c<\eta_d$; however, we constrain $\eta_c$ and $\eta_d$ as in~\eqref{eqn:batt_end} for a more practical roundtrip ESS-to-grid efficiency. While there exist more complex models that capture nonlinear ESS dynamics, we specifically consider this linear ESS model since it is widely used for grid-connected ESS applications \cite{jin2017foresee,Zarrilli2018,kkt_opf2013,7792609,baker2017optEss}.
	
	{The available solar power generated from PV, denoted $P_{\rm sol}^{(t)}$, is a function of the solar irradiance, area of the PV array, and array and inverter efficiency. The curtailed solar power is denoted $\Pc^{(t)}$. The overall power consumption from building loads is denoted $P_{\rm L,house}^{(t)}\geq0$. The power balance is given by:
		\begin{align}
		0=-P_{\rm grid}^{(t)}+P_{\rm L,house}^{(t)}-(P_{\rm sol}^{(t)}-\Pc^{(t)})-P_{\rm dis}^{(t)}+P_{\rm ch}^{(t)}.
		\end{align}
		
		Next, we provide the overall MPC-based optimization problem for the HEMS. The objective function $f_{\rm cost}(\underline{\mathbf{x}}_{N})$, which captures both customer electricity price sensitivity and ESS lifetime considerations, is given by:
		\begin{align}
		f_{\rm cost}(\underline{\mathbf{x}}_{N}) = \sum_{t=1}^{N}\left(c_e^{(t)}\Pg^{(t)}+\alpha\Pch^{(t)}+\beta\Pd^{(t)}\right), \label{eqn:fcost_tou}
		\end{align}
		where $N$ denotes the prediction horizon, $c_e^{(t)}>0$ for all $t\in\{1,2,\dots,N\}$ {reflects the unit cost of energy from the grid, and $\alpha,\beta\geq0$ reflect the lifetime cost of charging and discharging the battery, respectively}. Other ESS lifetime considerations are presented in \cite{batt_life,Zarrilli2018}. We allow the electricity drawn from the grid to be subject to a time-varying pricing schedule $c_e^{(t)}$, to include situations such as TOU pricing. The control variables at each time $t$ are collected in the vector $\underline{\mathbf{x}}_N=[	\mathbf{x}^{(1)} \; \mathbf{x}^{(2)} \; \dots \;\mathbf{x}^{(N)}]^{\rm T}$ where $\mathbf{x}^{(t)}~=~[
		P_{\rm grid}^{(t)} \; P_{\rm ch}^{(t)}\;P_{\rm dis}^{(t)}\;\Pc^{(t)}]^{\rm T}$.} Based on these discussions, the overall linear optimization program is:
	\begin{subequations}
		\begin{align} 
		(\mathcal{P}_1)\quad&\underset{\underline{\mathbf{x}}_{N}}{\min} \qquad  f_{\rm cost}(\underline{\mathbf{x}}_{N}) \label{eqn:opt_begin}\\
		& \text{subject to} \nonumber\\
		0 &\leq P_{\rm grid}^{(t)}, \label{eqn:opt_pgrid}\\
		0& \leq P_{\rm ch}^{(t)} \leq P_{\rm ch,max},  \label{eqn:opt_pch}\\
		0& \leq P_{\rm dis}^{(t)} \leq P_{\rm dis,max},  \label{eqn:opt_pdis}\\
		E^{(t+1)}&= E^{(t)} + \eta_c \Delta tP_{\rm ch}^{(t)} -\eta_d\Delta t P_{\rm dis}^{(t)},  \label{eqn:opt_battSOC}\\ 
		E_{\rm min}& \leq E^{(t+ 1)}\leq E_{\rm max},  \label{eqn:opt_battSOClims}\\
		0&\leq \Pc^{(t)}\leq \Psol^{(t)}, \label{eqn:opt_curtlims}\\
		0&=-P_{\rm grid}^{(t)}+P_{\rm L,house}^{(t)}-(P_{\rm sol}^{(t)}-\Pc^{(t)})-P_{\rm dis}^{(t)}+P_{\rm ch}^{(t)}, \label{eqn:opt_end}
		\end{align} 
	\end{subequations}
	for all $t=\{1,\dots,N\}$. The MPC-based HEMS optimization problem given in~\eqref{eqn:opt_begin}-\eqref{eqn:opt_end}, denoted ($\mathcal{P}_1$), will be the basis of the analysis presented in Section~\ref{sec:proofs} for the linear ESS model being considered. We will omit the constraint in~\eqref{eqn:opt_pgrid} when addressing customer price sensitivity under net metering. Note that the above optimization problem is always feasible as the (traditional operating) point $\xvec_N$ given by $\Pd^{(t)}=\Pch^{(t)}=0$, $\Pc^{(t)}=\Psol^{(t)}$, and $\Pg^{(t)}=P_{\rm L,house}^{(t)}$ for all $t$ is a feasible point for this optimization problem. 
	\vspace{-0.3cm}
	\section{ESS Model Behavior} \label{sec:proofs}
	Next, we derive the KKT conditions for the optimization problem given in ($\mathcal{P}_1$) and we prove that simultaneous charging and discharging in the ESS model \eqref{eqn:batt_begin}-\eqref{eqn:batt_end} is suboptimal.
	\subsection{KKT Conditions}
	The Karush-Kuhn-Tucker (KKT) conditions provide necessary and, in some general settings, sufficient conditions which must be satisfied for the solution of a broad range of optimization problems to be optimal. These conditions will be leveraged to show, under various system conditions and pricing schemes, that simultaneous charging and discharging of the ESS model is suboptimal (i.e.,  solutions of this form violate the KKT conditions).

	To aid in the presentation of the KKT conditions, we introduce additional notation for describing the inequality and equality constraints in the optimization problem given in ($\mathcal{P}_1$). Notice that each of the inequality constraints in \eqref{eqn:opt_pgrid}-\eqref{eqn:opt_curtlims} describes $N$ inequalities (one for each time $t\in\{1,\dots,N\}$) for each of the upper and lower bounds on each of the control variables and the ESS state of charge. Writing ($\mathcal{P}_1$) in standard form \cite{Boyd2004}, the functions on the left-hand side of the inequality constraints are denoted:
	\begin{align*}
	&\fgridl^{t}(\underline{\mathbf{x}}_{N})=-P_{\rm grid}^{(t)},\\
	&\fchl^{t}(\underline{\mathbf{x}}_{N})=-P_{\rm ch}^{(t)},\\
	&\fchu^{t}(\underline{\mathbf{x}}_{N})=P_{\rm ch}^{(t)}-P_{\rm ch,max},\\
	&\fdisl^{t}(\underline{\mathbf{x}}_{N})=-P_{\rm dis}^{(t)},\\
	&\fdisu^{t}(\underline{\mathbf{x}}_{N})=P_{\rm dis}^{(t)}-P_{\rm dis,max},\\
	&\fsocl^{t}(\underline{\mathbf{x}}_{N})=E_{\rm min}- E^{0} + \Delta t \sum_{n=1}^{t}\left(\eta_d P_{\rm dis}^{(n)}-\eta_c P_{\rm ch}^{(n)}\right),\\
	&\fsocu^{t}(\underline{\mathbf{x}}_{N})=E^{0} + \Delta t \sum_{n=1}^{t}\left(\eta_c P_{\rm ch}^{(n)} -\eta_d P_{\rm dis}^{(n)}\right)-E_{\rm max},\\
	&\fsoll^{t}(\underline{\mathbf{x}}_{N})=-\Pc^{(t)},\\
	&\fsolu^{t}(\underline{\mathbf{x}}_{N})=\Pc^{(t)}-\Psol^{(t)},
	\end{align*}
	for all $t\in\{1,\dots,N\}$ where $E^{0}$ is the initial ESS SOC. Similarly, the equality constraint in \eqref{eqn:opt_end} describes $N$ equality constraints. We denote the $N$ expressions on the right-hand-side of the equality constraint \eqref{eqn:opt_end} by the function $h^{t}(\underline{\mathbf{x}}_{N})$ for $t\in\{1,2,\dots,N\}$.
	
	Next, we derive the KKT conditions for the optimization problem given in ($\mathcal{P}_1$). Let $\underline{\mathbf{\tilde{x}}}_{N}$ denote an optimal solution for the problem. Then the KKT conditions are:
	\begin{subequations}
		\begin{align}
		&\text{Primal Feasibility:}\nonumber \\
		&\fgridl^{t}(\underline{\mathbf{\tilde{x}}}_{N}) \leq 0, \label{eqn:K1a_g}\\
		&\fchl^{t}(\underline{\mathbf{\tilde{x}}}_{N}),\fdisl^{t}(\underline{\mathbf{\tilde{x}}}_{N}),\fsocl^{t}(\underline{\mathbf{\tilde{x}}}_{N})\leq 0,  \label{eqn:K1a_battl}\\
		&\fchu^{t}(\underline{\mathbf{\tilde{x}}}_{N}),\fdisu^{t}(\underline{\mathbf{\tilde{x}}}_{N}),\fsocu^{t}(\underline{\mathbf{\tilde{x}}}_{N}) \leq 0,  \label{eqn:K1a_battu}\\
		&\fsoll^{t}(\underline{\mathbf{\tilde{x}}}_{N}),\fsolu^{t}(\underline{\mathbf{\tilde{x}}}_{N}) \leq 0,  \label{eqn:K1a_sol}\\
		&h^{t}(\underline{\mathbf{\tilde{x}}}_{N}) = 0, \label{eqn:K1b}\\
		&\text{Dual Feasibility:}\nonumber\\
		&\lgridl^{t}\geq 0,  \label{eqn:K2_g} \\
		& \lchl^t,\lchu^t,\ldisl^t,\ldisu^t,\lsocl^t,\lsocu^t \geq 0, \label{eqn:K2_batt} \\
		&\lsoll^t,\lsolu^t \geq 0, \label{eqn:K2_sol} \\
		&\text{Complementary Slackness:}\nonumber \\
		&\lgridl^t\fgridl^t(\underline{\mathbf{\tilde{x}}}_{N}) = 0, \label{eqn:K3_g}\\	
		&\lchl^t\fchl^t(\underline{\mathbf{\tilde{x}}}_{N})=\ldisl^t\fdisl^t(\underline{\mathbf{\tilde{x}}}_{N})=\lsocl^t\fsocl^t(\underline{\mathbf{\tilde{x}}}_{N}) = 0, \label{eqn:K3_battl}\\	
		&\lchu^t\fchu^t(\underline{\mathbf{\tilde{x}}}_{N})=\ldisu^t\fdisu^t(\underline{\mathbf{\tilde{x}}}_{N})=\lsocu^t\fsocu^t(\underline{\mathbf{\tilde{x}}}_{N}) = 0, \label{eqn:K3_battu}\\
		&\lsoll^t\fsoll^t(\underline{\mathbf{\tilde{x}}}_{N})=\lsolu^t\fsolu^t(\underline{\mathbf{\tilde{x}}}_{N}) = 0, \label{eqn:K3_sol}\\
		&\text{Stationarity:}\nonumber \\	
		&\nabla f_{\rm cost}(\underline{\mathbf{\tilde{x}}}_{N}) + \sum_{t=1}^{N}\bigg(\lgridl^t\nabla \fgridl^t(\underline{\mathbf{\tilde{x}}}_{N})+\lchl^t\nabla \fchl^t(\underline{\mathbf{\tilde{x}}}_{N})+\nonumber\\
		&\qquad \lchu^t\nabla \fchu^t(\underline{\mathbf{\tilde{x}}}_{N})+\ldisl^t\nabla \fdisl^t(\underline{\mathbf{\tilde{x}}}_{N})+\ldisu^t\nabla \fdisu^t(\underline{\mathbf{\tilde{x}}}_{N}) + \nonumber \\
		& \qquad \lsocl^t\nabla\fsocl^t(\underline{\mathbf{\tilde{x}}}_{N})+\lsocu^t\nabla\fsocu^t(\underline{\mathbf{\tilde{x}}}_{N})+\lsoll^t\nabla\fsoll^t(\underline{\mathbf{\tilde{x}}}_{N}) + \nonumber\\
		& \qquad\lsolu^t\nabla\fsolu^t(\underline{\mathbf{\tilde{x}}}_{N}) +\mu_{t}\nabla h^t(\underline{\mathbf{\tilde{x}}}_{N})\bigg) = \mathbf{0},  \label{eqn:K4}
		\end{align}
	\end{subequations}
	where $\mathbf{0}\in \mathbb{R}^{4N}$, and $\underaccent{\bar}{\lambda}_{\rm const}^{t}/\bar{\lambda}_{\rm const}^{t}$ is the Lagrange multiplier associated with the inequality constraint $\underaccent{\bar}{f}_{\rm const}^{t}(\xopt_H)/\bar{f}_{\rm const}^{t}(\xopt_H)$, respectively, and $\mu_t$ is the Lagrange multiplier associated with the power balance equality constraint $h^t(\xopt_H)$ at time $t$ \cite{Boyd2004,Bertsekas2003}. Then, the condition in \eqref{eqn:K4} is decomposed into:
	\begin{subequations}
		\begin{align}
		&c_e^{(t)} - \lgridl^t - \mu_{t}=0, \label{eqn:dL1}\\
		&\alpha-\lchl^{t} + \lchu^{t} + \eta_c\Delta t\sum_{n=t}^{N}(\lsocu^{n}-\lsocl^n)+\mu_{t}=0,\label{eqn:dL2}\\
		&\beta-\ldisl^{t} + \ldisu^t + \eta_d\Delta t\sum_{n=t}^{N}(\lsocl^n-\lsocu^n)-\mu_{t}=0, \label{eqn:dL3}\\
		&-\lsoll^t+\lsolu^t+\mu_{t}=0, \label{eqn:dL4}
		\end{align}
	\end{subequations}
	for all $t\in\{1,\dots, N\}$. For convex optimization problems with differentiable objective functions, any solution that satisfies the KKT conditions given in \eqref{eqn:K1a_g}-\eqref{eqn:K4} is optimal \cite{Boyd2004}. Thus, we leverage the KKT conditions to show that feasible points to the MPC-based HEMS optimization problem given in ($\mathcal{P}_1$) where the ESS is simultaneously charging and discharging are suboptimal in certain situations. \vspace{-0.3cm}
	\subsection{Convex ESS Model Behavior Guarantees}
	Next, we show that feasible points where the ESS is charging and discharging simultaneously ($0<\Pch^{(t)}\leq\Pmaxc$, $0<\Pd^{(t)}\leq \Pmaxd$) give rise to suboptimal solutions. Further, we show that a solution with simultaneous charging and discharging is suboptimal for an objective function that captures customer energy cost sensitivity and ESS lifetime considerations. The customer energy cost sensitivity is studied for two cases: 1) the case where power is not allowed to be exported back to the grid, and 2) the net metering case. Under net metering, a consumer can provide a grid service by exporting excess available solar power to offset the price of power drawn from the grid, thus the constraint in~\eqref{eqn:opt_pgrid} is omitted in ($\mathcal{P}_1$). This implies $\tilde{P}_{\rm grid}^{(t)}$ can be negative, which represents the customer's ability to export excess solar power to the grid. In each case, we assume that the optimal solution for some time $\tau\in\{1,\ldots,N\}$ in the prediction horizon is $\tilde{\mathbf{x}}^{(\tau)}~=[
	\tilde{P}_{\rm grid}^{(\tau)}\;\tilde{P}_{\rm ch}^{(\tau)}\;\tilde{P}_{\rm dis}^{(\tau)}\;\tPc^{(\tau)}
	]^{\rm T}$, where $0<\tilde{P}_{\rm ch}^{(\tau)}\leq\Pmaxc$ and $0<\tilde{P}_{\rm dis}^{(\tau)}\leq P_{\rm dis,max}$, from which we obtain a contradiction to show that a solution with simultaneous ESS charging and discharging is suboptimal. In Proposition~\ref{thm:ce_tou_sub}, we show that a solution to ($\mathcal{P}_1$) with simultaneous ESS charging and discharging is suboptimal when customers are not able to export excess power onto the grid.
		\begin{proposition} \label{thm:ce_tou_sub}
			Assume $\alpha,\beta\geq 0$, $(\alpha+\beta)>0$, $0<\eta_c<1<\eta_d$, and $c_e^{(t)}>0$ for all $t$. A solution to ($\mathcal{P}_1$) where the ESS simultaneously charges and discharges, i.e. $0<\tilde{P}_{\rm ch}^{(t)}\leq \Pmaxc$ and $0<\tilde{P}_{\rm dis}^{(t)}\leq \Pmaxd$ for any $t~\in~\{1,\dots,N\}$, is suboptimal if customers are not able to export excess power onto the grid, i.e. $\Pg^{(t)}\geq0$.
	\end{proposition}
	\begin{proof}
		We prove this claim by contradiction. So, assume at some time $\tau~\in~\{1,\dots,N\}$ the optimal solution  to ($\mathcal{P}_1$) is $\tilde{\mathbf{x}}^{(\tau)}~=~[\tilde{P}_{\rm grid}^{(\tau)}\;\tilde{P}_{\rm ch}^{(\tau)}\;\tilde{P}_{\rm dis}^{(\tau)}\;\tPc^{(\tau)}]^{\rm T}$, where $0~\leq~\tPc^{(\tau)}~\leq~\Psol^{(\tau)}$, $0~<~\tilde{P}_{\rm ch}^{(\tau)}\leq P_{\rm ch,max}$, and $0<\tilde{P}_{\rm dis}^{(\tau)}\leq\Pmaxd$, i.e., it is optimal to simultaneously charge and discharge the ESS. The objective function is given in~\eqref{eqn:fcost_tou} where $c_e^{(t)}>0$, $\alpha,\beta\geq 0$, and $(\alpha+\beta)>0$ for all $t$. 
		
		We consider the situation where excess power cannot be exported onto the grid ($\tilde{P}_{\rm grid}^{(\tau)}\geq0$). We will prove this by looking at the following two solutions for $\tPgrid^{(\tau)}$: (i) $0<\tilde{P}_{\rm grid}^{(\tau)}$, and (ii) $0=\tilde{P}_{\rm grid}^{(\tau)}$.
		
	First, we consider Case (i) when the optimal solution at time $\tau$ is such that $0<\tilde{P}_{\rm grid}^{(\tau)}$, $\tilde{P}_{\rm ch}^{(\tau)},\tilde{P}_{\rm dis}^{(\tau)}$, and $0\leq\tPc^{(\tau)}\leq\Psol^{(\tau)}$. We leverage the complementary slackness conditions at time $\tau$ in~\eqref{eqn:K3_g}-\eqref{eqn:K3_sol}, therefore, $\lgridl^{\tau}=\lchl^{\tau}=\ldisl^{\tau}=0$. We also determine that $\lchu^{\tau},\ldisu^{\tau},\lsoll^{\tau},\lsolu^{\tau}\geq 0$ due to the assumptions $\tilde{P}_{\rm ch}^{(\tau)}\leq P_{\rm ch,max}$ and $\tilde{P}_{\rm dis}^{(\tau)}\leq\Pmaxd$, the solar power curtailment satisfies~\eqref{eqn:opt_curtlims}, and the dual feasibility condition. Using the above conditions on the Lagrange multipliers at time $\tau$ together with~\eqref{eqn:dL1}-\eqref{eqn:dL4}, we have: 
		\begin{align}
		&c_e^{(\tau)} = \mu_{\tau}, \\
		&\alpha + \lchu^{\tau} + \eta_c\Delta t\sum_{n=t}^{N}(\lsocu^{n}-\lsocl^n)+c_e^{(\tau)}=0, \label{eqn:dL2_c1}\\
		&\beta+ \ldisu^{\tau} + \eta_d\Delta t\sum_{n=t}^{N}(\lsocl^{n}-\lsocu^n)-c_e^{(\tau)}=0,\label{eqn:dL3_c1}\\
		&-\lsoll^\tau+\lsolu^\tau+\mu_{\tau}=0. \label{eqn:dL4_c1}
		\end{align}
		Solving for $\mathcal{I}=\sum_{n=t}^{N}(\lsocu^{n}-\lsocl^n)$ in~\eqref{eqn:dL2_c1}, and then replacing $-\mathcal{I}$ in~\eqref{eqn:dL3_c1}, we get:
		\begin{eqnarray}
		\left(\beta + \frac{\eta_d}{\eta_c}\alpha\right)+\ldisu^{\tau}+\frac{\eta_d}{\eta_c}\lchu^{\tau}+\left(\frac{\eta_d}{\eta_c}-1\right)c_e^{(\tau)} = 0. \label{eqn:c1_cont}
		\end{eqnarray}
		Note that by dual feasibility condition \eqref{eqn:K2_batt}, the second and third terms in \eqref{eqn:c1_cont} are nonnegative. The first term in parenthesis and the fourth term in~\eqref{eqn:c1_cont} are strictly positive since we assumed $\alpha,\beta\geq0$, $(\alpha+\beta)>0$, $c_e^{(\tau)}>0$, and $0~<~\eta_c~<~1<\eta_d$. Thus, we obtain a contradiction since the left hand side of~\eqref{eqn:c1_cont} cannot equal 0. Therefore, in a situation where excess power cannot be exported to the grid, a solution to ($\mathcal{P}_1$) where $\tilde{P}_{\rm grid}^{(t)}>0$ and the ESS is simultaneously charging and discharging for some $t\in\{1,\dots,N\}$ is suboptimal because it does not satisfy the KKT conditions given in~\eqref{eqn:K1a_g}-\eqref{eqn:K4}. 
		
		Next, we consider Case (ii) where the optimal solution is such that $\tPgrid^{(\tau)}=0$, $0~<~\tilde{P}_{\rm ch}^{(\tau)}\leq P_{\rm ch,max}$, and $0<\tilde{P}_{\rm dis}^{(\tau)}\leq\Pmaxd$ at time $\tau\in\{1,\dots,N\}$. This case captures the situation where there is high PV penetration, i.e. $\Psol^{(\tau)}>P_{\rm L,house}^{(t)}$ and $\tPgrid^{(\tau)}=0$. We break this case further into two cases with respect to curtailment:
		\begin{enumerate}
			\item $0\leq \tPc^{(\tau)}<\Psol^{(\tau)}$, and
			\item $\tPc^{(\tau)}=\Psol^{(\tau)}$. 
		\end{enumerate}
		In Case 1, using the complementary slackness conditions at time $\tau$ in~\eqref{eqn:K3_g}-\eqref{eqn:K3_sol}, we determine that $\lgridl^\tau=\lchl^\tau=\ldisl^\tau=\lsolu^\tau=0$. Then, using \eqref{eqn:dL4}, we obtain:
		\begin{align*}
		\mu_\tau = \lsoll^\tau,
		\end{align*}
		where $\lsoll^\tau\geq0$ by dual feasibility in \eqref{eqn:K2_sol}. Thus, $\mu_\tau\geq0$, and we recover \eqref{eqn:c1_cont} again by combining \eqref{eqn:dL2_c1} and \eqref{eqn:dL3_c1}. We obtain a contradiction since \eqref{eqn:c1_cont} does not hold as the left hand side is strictly greater than zero. Therefore, in Case 1, a solution such that $\tPgrid^{(\tau)}=0$, $0\leq \tPc^{(\tau)}<\Psol^{(\tau)}$, and the ESS is simultaneously charging and discharging is suboptimal.
		
		Next, we consider Case 2 where the optimal solution $\xopt^{(\tau)}$ at time $\tau\in\{1,\dots,N\}$ is such that $\tPgrid^{(\tau)}=0$, $\tPc^{(\tau)}=\Psol^{(\tau)}$ and the ESS is simultaneously charging and discharging. In this case, we will use the primal variables. Using the power balance in \eqref{eqn:opt_end} at time $\tau$, we obtain: 
		\begin{align*}
		P_{\rm L, house}^{(\tau)} = \tPdis^{(\tau)}-\tPch^{(\tau)},
		\end{align*}
		which implies $\tPdis^{(\tau)}\geq\tPch^{(\tau)}>0$ since we assume $P_{\rm L, house}^{(\tau)}\geq0$ and the ESS is simultaneously charging and discharging at time $\tau$. Thus, the solution at time $\tau$ in $\xopt_N$ is of the form:
		\begin{align*}
		\xopt^{(\tau)}&=[\tPgrid^{(\tau)}\;\tPch^{(\tau)}\;\tPdis^{(\tau)}\;\tPc^{(\tau)}]^{\rm T}\\
		&= [0\;\tPch^{(\tau)}\;\tPdis^{(\tau)}\;\Psol^{(\tau)}]^{\rm T},
		\end{align*} 
		where $\tPdis^{(\tau)}\geq\tPch^{(\tau)}>0$. Without loss of generality, assume that time $\tau$ is the last time there is simultaneous charging and discharging.
Let $X^*$ be the set of all solutions of ($\probl$). 
Finally, for a solution $\xvec_N\in X^*$, define its total charging $\Pch(\xvec_N)$ to be: 
		\begin{align*}
		\Pch(\xvec_N)=\sum_{t=1}^N\Pch^{(t)}.
		\end{align*}
		Note that $X^*$ is a closed and bounded set and hence, a solution with minimum total charging always exists, i.e.\ the set $\argmin_{\xvec_N\in X^*}\Pch(\xvec_N)$ is non-empty. Let $\xopt_N\in \argmin_{\xvec_N\in X^*}\Pch(\xvec_N)$. 		
		Let $\delta=\tPdis^{(\tau)}-\tPch^{(\tau)}\geq0$. First, let us assume that for all time $t>\tau$, $\tPch^{(t)}=0$. In this case, define $\xhat_N\in \R^{4N}$ by: 
		\begin{align}\label{eqn:bhxdef}
		\bhx^{(t)} = 
		\begin{cases}
		\btx^{(t)}&t\not=\tau\\
		[0\quad 0\quad \delta\quad \Psol^{(t)}]^{\rm T} & t=\tau
		\end{cases},
		\end{align}
		i.e., $\xopt_N$ and $\xhat_N$ are identical except that $\hPch^{(\tau)}=0$ and $\hPdis^{(\tau)}=\delta$. Note that since, $\xhat_N=\xopt_N$ for all $t\neq\tau$, $\xhat_N$ satisfies \eqref{eqn:opt_pgrid}, \eqref{eqn:opt_pch}, \eqref{eqn:opt_pdis}, and \eqref{eqn:opt_end} for all $t\not=\ts$. Similarly, for $t=\tau$, since $\xopt_N$ is a feasible point and $\hPch^{(\tau)}-\hPdis^{(\tau)}=\tPch^{(\tau)}-\tPdis^{(\tau)}$, \eqref{eqn:opt_pgrid} and \eqref{eqn:opt_end} hold for $t=\tau$. Using \eqref{eqn:opt_battSOC}, if we let $\st{\hat{E}}$ be the state of charge for the solution $\xhat_N$, then \eqref{eqn:opt_battSOClims} is satisfied for $t<\tau$ since $\bhx^{(t)}=\btx^{(t)}$ for all $t<\tau$. For time $t=\tau$, we have: 
		\begin{align}\label{eqn:SOCreduce}
		\hat{E}^{(t+1)}&= \hat{E}^{(t)} + \eta_c \Delta t\hat{P}_{\rm ch}^{(t)} -\eta_d\Delta t \hat{P}_{\rm dis}^{(t)}\cr 
		&=\tilde{E}^{(t)}- \eta_d\Delta t \delta\cr 
		&=\tilde{E}^{(t)}+\eta_c \Delta t\tilde{P}_{\rm ch}^{(t)} -\eta_d\Delta t \tilde{P}_{\rm dis}^{(t)}+(\eta_d-\eta_c)\Delta t \tilde{P}_{\rm ch}^{(t)}\cr 
		&=\tilde{E}^{(t+1)}+(\eta_d-\eta_c)\Delta t \tilde{P}_{\rm ch}^{(t)}.
		\end{align}
		The feasibility of $\xopt_N$ and the equalities above imply that: 
		\begin{align*}
		\Emin\leq \tilde{E}^{(t+1)}<\hat{E}^{(t+1)}\leq\tilde{E}^{(t)}\leq E_{\max}.
		\end{align*}
		Thus, at time $t=\tau$, constraints~\eqref{eqn:opt_pgrid}-\eqref{eqn:opt_end} are met. Since we assumed that $\tPch^{(t)}=0$ for $t>\tau$, we recursively have:
		\begin{align*}
		\Emin\leq \tilde{E}^{(t+1)}<\hat{E}^{(t+1)}\leq \hat{E}^{(\tau+1)}<\tilde{E}^{(\tau)}\leq E_{\max}.
		\end{align*}
		Therefore, for all time $t$, all the constraints \eqref{eqn:opt_pgrid}-\eqref{eqn:opt_end} hold for the vector $\xhat_N$, and hence, $\xhat_N$ is a feasible point. On the other hand, by \eqref{eqn:fcost_tou}, we have 
		\begin{align*}
		f_{\rm cost}({\xhat_N})=f_{\rm cost}({\xopt_N})-(\alpha+\beta)\tPch^{(\tau)}< f_{\rm cost}({\xopt_N}),
		\end{align*}
		because $\tPch^{(\tau)}>0$, $\alpha,\beta\geq0$, and $(\alpha+\beta)>0$. Thus, by our construction, we obtain a contradiction to the fact that $\xopt_N\in \argmin_{\xvec_N\in X^*}\Pch(\xvec_N)$.
		  
		For Case 2, it remains to consider the case that $\tPch^{(t)}>0$ for some $t>\tau$. Let $\ts$ be the first time after $\tau$ such that $\tPch^{(\ts)}>0$. In this case, let
		\begin{align}
		p=\min\left(\frac{\eta_c\tPch^{(\ts)}}{\eta_d-\eta_c},\tPch^{(\tau)}\right).
		\end{align}
		Similar to \eqref{eqn:bhxdef}, define: $\bhx^{(t)} =$
		\begin{align}\label{eqn:bhxdef2}
		\begin{cases}
		\btx^{(t)}&t\not=\tau,\ts\\
		[\tPgrid^{(\tau)},\; \tPch^{(\tau)}-p,\; \tPdis^{(\tau)}-p,\; \tPc^{(\tau)}]^{\rm T} & t=\tau\\
		[\tPgrid^{(\ts)}-a,\; \tPch^{(\ts)}-(\frac{\eta_d}{\eta_c}-1) p,\; 0,\; \tPc^{(\ts)}+b]^{\rm T}&t=\ts
		\end{cases}
		\end{align}
		where $a\in [0,\tPgrid^{(\ts)}]$, $b\in [0,\Psol^{(\ts)}-\tPc^{(\ts)}]$, such that $a~+~b~=~(\frac{\eta_d}{\eta_c}-1)p$. To show that such an $a,b$ exist, using the power balance equation in \eqref{eqn:opt_end}, we have:
		\begin{align*}
		\tPgrid^{(\ts)}+(\Psol^{(\ts)}-\tPc^{(\ts)})&\geq \tPch^{(\ts)}\geq \textstyle(\frac{\eta_d}{\eta_c}-1)p,
		\end{align*}
		where the first inequality follows as $P_{\rm L,house}^{(\ts)}\geq 0$, and the second inequality follows from the choice of $p\leq \frac{\eta_c\tPch^{(\ts)}}{\eta_d-\eta_c}$. {Therefore, such an $a$ and $b$ as defined above exists.}
		Next, we show $\bhx_N$ in~\eqref{eqn:bhxdef2} satisfies the constraints in ($\mathcal{P}_1$): 
		\begin{enumerate}[a.]
			\item Constraint \eqref{eqn:opt_pgrid}: Note that $\hPgrid^{(t)}=\tPgrid^{(t)}\geq 0$ for all $t\neq \ts$. For $t=\ts$, since $a\in [0,\tPgrid^{(\ts)}]$, it follows that $0~\leq~\tPgrid^{(\ts)}-a=\hPgrid^{(\ts)}$. Thus, $\hPgrid^{(t)}$ satisfies \eqref{eqn:opt_pgrid} $\forall t$.
			\item{ Constraint \eqref{eqn:opt_pch}: Note that $\hPch^{(t)}=\tPch^{(t)}$ for all $t\neq\tau,\ts$ and $0\leq p\leq \tPch^{(\tau)}$. Then, $0\leq  \tPch^{(\tau)}-p\leq \tPch^{(\tau)}\leq P_{\max}$ for $t=\tau$, and we have $0\leq (\frac{\eta_d}{\eta_c}-1)p\leq \tPch^{(\ts)}\leq P_{\max}$ for $t=\ts$. Therefore, $\hPch^{(t)}$ satisfies \eqref{eqn:opt_pch} $\forall t$.} 
			
			\item {Constraint \eqref{eqn:opt_pdis}: Recall $\hPdis^{(t)}=\tPdis^{(t)}$ $\forall t\neq\tau,\ts$. It is clear $\hPdis^{(\ts)}$ satisfies \eqref{eqn:opt_pdis}. For $t=\tau$,  $0\leq p\leq \tPch^{(\tau)}\leq  \tPdis^{(\tau)} $ (by assumption), and $0\leq  \tPdis^{(\tau)}-p=\hPdis^{(\tau)}\leq  \tPdis^{(\tau)}\leq P_{\max}$. Thus, $\hPdis^{(t)}$ satisfies \eqref{eqn:opt_pdis} $\forall t$.} 
			
			\item Constraint \eqref{eqn:opt_battSOClims}: Since $\btx^{(t)}=\bhx^{(t)}$ for $t<\tau$, \eqref{eqn:opt_battSOClims} holds for $t<\tau$. For $t=\tau$, similar to \eqref{eqn:SOCreduce} and using the definition of $\xhat_N$ in \eqref{eqn:bhxdef2}, we have 
			{\begin{align}\label{eqn:SOCreduce2}
				\hat{E}^{(\tau+1)}&= \hat{E}^{(\tau)} + \eta_c \Delta t\hat{P}_{\rm ch}^{(\tau)} -\eta_d\Delta t \hat{P}_{\rm dis}^{(\tau)}\cr 
				&=\tilde{E}^{(t)}+\eta_c \Delta t(\tilde{P}_{\rm ch}^{(\tau)} -p)-\eta_d\Delta t (\tilde{P}_{\rm dis}^{(\tau)}-p)\cr 
				&=\tilde{E}^{(t+1)}+(\eta_d-\eta_c)\Delta t p.
				\end{align}
				By the main assumption of Case 2, i.e. $\tPdis^{(\tau)}\geq\tPch^{(\tau)}$, we have $\hPdis^{(\tau)}>\hPch^{(\tau)}$. Using this fact and the first equality in \eqref{eqn:SOCreduce2}, we have $\hat{E}^{(\tau+1)}\leq \hat{E}^{(\tau)}= \tilde{E}^{(\tau)}\leq \Emax$. From the last equality in \eqref{eqn:SOCreduce2}, we get $0\leq \tilde{E}^{(\tau+1)}\leq \hat{E}^{(\tau+1)}$. Then, overall, we have:
				\begin{align}\label{eqn:SOCtstar}
				0&\leq \tilde{E}^{(\tau+1)}\cr 
				&\leq\tilde{E}^{(\tau+1)}+ (\eta_d-\eta_c)\Delta t p\cr 
				&=\hat{E}^{(\tau+1)}\leq \hat{E}^{(\tau)}= \tilde{E}^{(\tau)}\leq \Emax.
				\end{align}} 
			Therefore, \eqref{eqn:opt_battSOClims} holds for $t=\tau$. For $t=\tau+1,\ldots,\ts-1$, we have $\hPch^{(t)}=\tPch^{(t)}=0$ and hence, 
			\begin{align*}
			\hat{E}^{(t+1)}=\hat{E}^{(t)}-\eta_d\Delta t \hat{P}_{\rm dis}^{(t)},
			\end{align*}
			and inductively, we can show that 
			\[\Emax\geq \hat{E}^{(\tau+1)}\geq \hat{E}^{(\tau+2)}\geq \cdots \geq \hat{E}^{(\ts)},\]
			and $\hat{E}^{(\tau+1)}=\tilde{E}^{(\tau+1)}+(\eta_d-\eta_c)\Delta t p$. This implies that $\hat{E}^{(t+1)}\geq \tilde{E}^{(t+1)}\geq 0$ as $\btx^{(t)}$ $\forall t<(\ts-1)$ is a feasible solution. For $t=\ts$, we have:
				\begin{align}
				\hat{E}^{(\ts+1)}&=\hat{E}^{(\ts)}+\eta_c\Delta t \hat{P}_{\rm ch}^{(\ts)}\cr 
				&=\tilde{E}^{(\ts)}+(\eta_d-\eta_c)\Delta t p + \cr
				&\qquad\qquad\eta_c\Delta t\left(\tPch^{(\ts)}-\left(\frac{\eta_d}{\eta_c}-1\right)p\right)\cr 
				&=\tilde{E}^{(\ts)}+\eta_c\Delta t\tPch^{(\ts)}\cr 
				&=\tilde{E}^{(\ts+1)}.	
				\end{align}
				Since $\tilde{E}^{(\ts+1)}$ satisfies \eqref{eqn:opt_battSOClims}, then $\hat{E}^{(\ts+1)}$ also satisfies \eqref{eqn:opt_battSOClims}. Additionally, since $\bhx^{(t)}=\btx^{(t)}$ for $t>\ts$, then \eqref{eqn:opt_battSOClims} also holds for $t>\ts$. Thus, $\hat{E}^{(t+1)}$ satisfies \eqref{eqn:opt_battSOClims} $\forall t$.
			\item  Constraint \eqref{eqn:opt_curtlims}: For this constraint, we have $\hPc^{(t)}=\tPc^{(t)}$ for all $t\neq \ts$. For $t=\ts$, we have $\hPc^{(\ts)}=\tPc^{(\ts)}+b$, where $b\in[0,\Psol^{(\ts)}-\tPc^{(\ts)}]$, which satisfies the condition \eqref{eqn:opt_curtlims} for $t=\ts$. {Thus, $\hPc^{(t)}$ satisfies \eqref{eqn:opt_pch} $\forall t$.}
			\item Constraint \eqref{eqn:opt_end}: The constraint \eqref{eqn:opt_end} is satisfied for all $t\neq\ts,\ts$ since $\bhx^{(t)}=\btx^{(t)}$ for all $t\neq\tau,\ts$. Notice that for $t=\tau$ and $t=\ts$, the vector $\xhat_N$ defined in \eqref{eqn:bhxdef2} is chosen such that \eqref{eqn:opt_end} is satisfied. Thus, the power balance equation is satisfied $\forall t$.
		\end{enumerate}
		Based on the above discussion, we have shown that $\xhat_N$ is a feasible solution to $(\probl)$. Also, we have: 
			\begin{align}
			\fcost(\xhat_N)&=\fcost(\xopt_N)-p\left(\alpha+\beta+\alpha\left(\frac{\eta_d}{\eta_c}-1\right)\right)-ac_e^{(\ts)}\nonumber\\
			&< \fcost(\xopt_N), 
			\end{align}
		since $\alpha,\beta\geq 0$ and $(\alpha+\beta)>0$. Thus, $\xhat_N\in X^*$. But $\Pch(\xhat_N)=\Pch(\xopt_N)-\frac{\eta_d}{\eta_c}p<\Pch(\xopt_N)$ which contradicts the fact that $\xopt_N$ is a minimizer of $\Pch(\cdot)$ among the solutions of $(\probl)$. Therefore, we conclude that such minimizers do not simultaneously charge and discharge.
\end{proof} 
For Prop. \ref{thm:ce_tou_sub}, notice that when $\alpha=\beta=0$, a solution with simultaneous ESS charging and discharging may be equivalent to a solution without simultaneous ESS charging and discharging behavior in certain situations. To ensure that there is an optimal solution without simultaneous ESS charging and discharging, we require $\alpha,\beta\geq0$ such that $(\alpha~+~\beta)>0$. Furthermore, if $\alpha=\beta=0$, from inspection of \eqref{eqn:c1_cont} in the proof of Prop.~\ref{thm:ce_tou_sub} we can see that simultaneous charging and discharging will be optimal when $\tPgrid^{(t)}=0$ and $0<\tPc^{(t)}\leq\Psol^{(t)}$. 

Next, we show that simultaneous charging and discharging is always suboptimal in a net metering situation for $(\mathcal{P}_1)$ where $\alpha,\beta\geq 0$. 
	\begin{proposition} \label{thm:ce_tou_nm}
		Assume $\alpha,\beta\geq 0$, $0<\eta_c<1<\eta_d$, and $c_e^{(t)}>0$ for all $t$. A solution to ($\mathcal{P}_1$) where the ESS simultaneously charges and discharges ($0<\tilde{P}_{\rm ch}^{(t)}\leq \Pmaxc$ and $0<\tilde{P}_{\rm dis}^{(t)}\leq \Pmaxd$ for any $t~\in~\{1,\dots,N\}$) is suboptimal if there is net metering where customers \textit{can} export excess power to the grid, i.e. $\Pg^{(t)}$ can be negative.
	\end{proposition}
\begin{proof}
	This proof is very similar to the proof of Prop.~\ref{thm:ce_tou_sub}. We again use contradiction to prove our claim. Assume that at time $\tau\in\{1,\dots,N\}$ the optimal solution is \linebreak$\tilde{\mathbf{x}}_{\tau}=[\tilde{P}_{\rm grid}^{(\tau)}\;\tilde{P}_{\rm ch}^{(\tau)}\;\tilde{P}_{\rm dis}^{(\tau)}\;\tPc^{(\tau)}]^{\rm T}$ where $0<\tilde{P}_{\rm ch}^{(t)}\leq \Pmaxc$, $0<\tilde{P}_{\rm dis}^{(t)}\leq \Pmaxd$, and $0~\leq~\tPc^{(\tau)}~\leq~\Psol^{(\tau)}$. We use the complementary slackness conditions at time $\tau$ in~\eqref{eqn:K3_battl}-\eqref{eqn:K3_sol}. For this case,~\eqref{eqn:K3_g} is omitted since constraint~\eqref{eqn:opt_pgrid} is excluded for net metering. Using similar analysis as in the proof of Prop.~\ref{thm:ce_tou_sub}, we again obtain~\eqref{eqn:c1_cont}, which leads to a contradiction since the $(\frac{\etad}{\etac}-1)c_e^t$ tern is strictly positive. Thus, under net metering, a solution to ($\mathcal{P}_1$) where the ESS simultaneously charges and discharges for any time $t\in\{1,\dots,N\}$ is suboptimal. \end{proof}
	Note that Propositions \ref{thm:ce_tou_sub} and \ref{thm:ce_tou_nm} still hold when electricity prices are constant, i.e. $c_e^{(t)}~>~0$ is constant for all $t=\{1,\dots,N\}$. For the case that an energy utility charges $c_e^{(t)}=\$0$/kWh for some time $t\in\{1,\dots,N\}$ to flatten grid power demand over time \cite{Spees2008}, we provide a remark on non-simultaneous ESS charging and discharging that follow directly from the propositions above.
	\begin{remark}
		In a situation where there is either net metering or excess power cannot be exported to the grid, if $c_e^{(t)}~=~\$0$/kWh for any $t\in\{1,\dots,N\}$, the results of Props.~\ref{thm:ce_tou_sub}-\ref{thm:ce_tou_nm} hold since we require $\alpha,\beta\geq0$ such that $(\alpha~+~\beta)~>~0$.
	\end{remark}
	The results from Props. \ref{thm:ce_tou_sub} and \ref{thm:ce_tou_nm} are summarized in Table \ref{tab:results}.
	\begin{table}[t!]
		\centering
		\caption{Proposition Results Summary for $\alpha\geq0$, $\beta\geq0$}
		\label{tab:results} \vspace{-0.2cm}
		\begin{tabular}{|c|c|c|c|}
			\hline
			$(\alpha + \beta)$ & $c_e^{(t)} \; \forall t$ & \begin{tabular}[c]{@{}c@{}}Net\\ Metering\end{tabular} & \begin{tabular}[c]{@{}c@{}}ESS Charging and\\ Discharging Behavior\end{tabular} \\ \hline
			$\geq 0$ &    $>0$    & Yes &      Non-simultaneous                      \\ \hline
			$>0$ &        $\geq0$       &    Yes      &    Non-simultaneous                       \\ \hline
			$=0$&           $=0$               &   Yes   &      Simultaneous            \\ \hline
		    $>0$&     $\geq0$     &        No             &   Non-simultaneous                \\ \hline
			$=0$&          $\geq 0$     &     No &      Simultaneous                    \\ \hline
		\end{tabular}\vspace{-0.5cm}
	\end{table}
	\section{Simulations: ESS Operation}\label{sec:sims}
	To demonstrate the charging and discharging behavior of the ESS model, we provide simulation results for the proposed MPC-based HEMS algorithm in \eqref{eqn:opt_begin}-\eqref{eqn:opt_end} to highlight each proposition in Section~\ref{sec:proofs}. The MATLAB-based modeling system for solving disciplined convex programs, CVX, is used in this work. The MPC-based HEMS optimization is solved in a receding horizon manner with a 24 hour prediction horizon and 1 hour time intervals. 
	
	The residential PV array size is $20 {\rm m}^{2}$ with a tilt of 30\si{\degree} and an efficiency of 16\%. The solar irradiance forecasts were obtained from NOAA USCRN data \cite{noaa_data}. The 5-kWh residential ESS system is restricted between 15\% and 85\% of the maximum SOC to preserve the ESS lifetime \cite{jin2017foresee}. The initial ESS SOC is 2~kWh. The ESS inverter power limit is 3~kW with an inverter efficiency of 95\%. The ESS charging efficiency $\eta_c$ is 95\%, and the discharging efficiency is $\eta_d=\frac{1}{\eta_c}$. 
	\subsection{Simulation Results}
	The proposed HEMS algorithm in~\eqref{eqn:opt_begin}-\eqref{eqn:opt_end} with the objective function in~\eqref{eqn:fcost_tou} is simulated for the case where the price of electricity is constant $c_e^{(t)}=\$0.11$/kWh $\forall t$, $\alpha=\beta=0$, and excess power is not able to be exported  to the grid. The simulation results for this case are provided in Fig.~\ref{fig:flat_ce}. From Fig.~\ref{fig:flat_ce} (top), we can see that $P_{\rm c}^{(t)}=0$ $\forall t$, and when $P_{\rm grid}^{(t)}=0$, the excess solar power from satisfying the total residential load is able to be stored in the ESS, thus allowing $\alpha,\beta=0$ for this case. Additionally, when $P_{\rm grid}^{(t)}>0$, the ESS discharges to contribute power to satisfying the total load. The ESS operating behavior in Fig.~\ref{fig:flat_ce} (middle) and (bottom) highlights that non-simultaneous ESS charging and discharging were optimal solutions at each time $t\in\{1,\dots, N\}$ with the convex ESS model in~\eqref{eqn:batt_begin}-\eqref{eqn:batt_end}. 
	\begin{figure}[t!]
		\centering
		\includegraphics[width=3.5in]{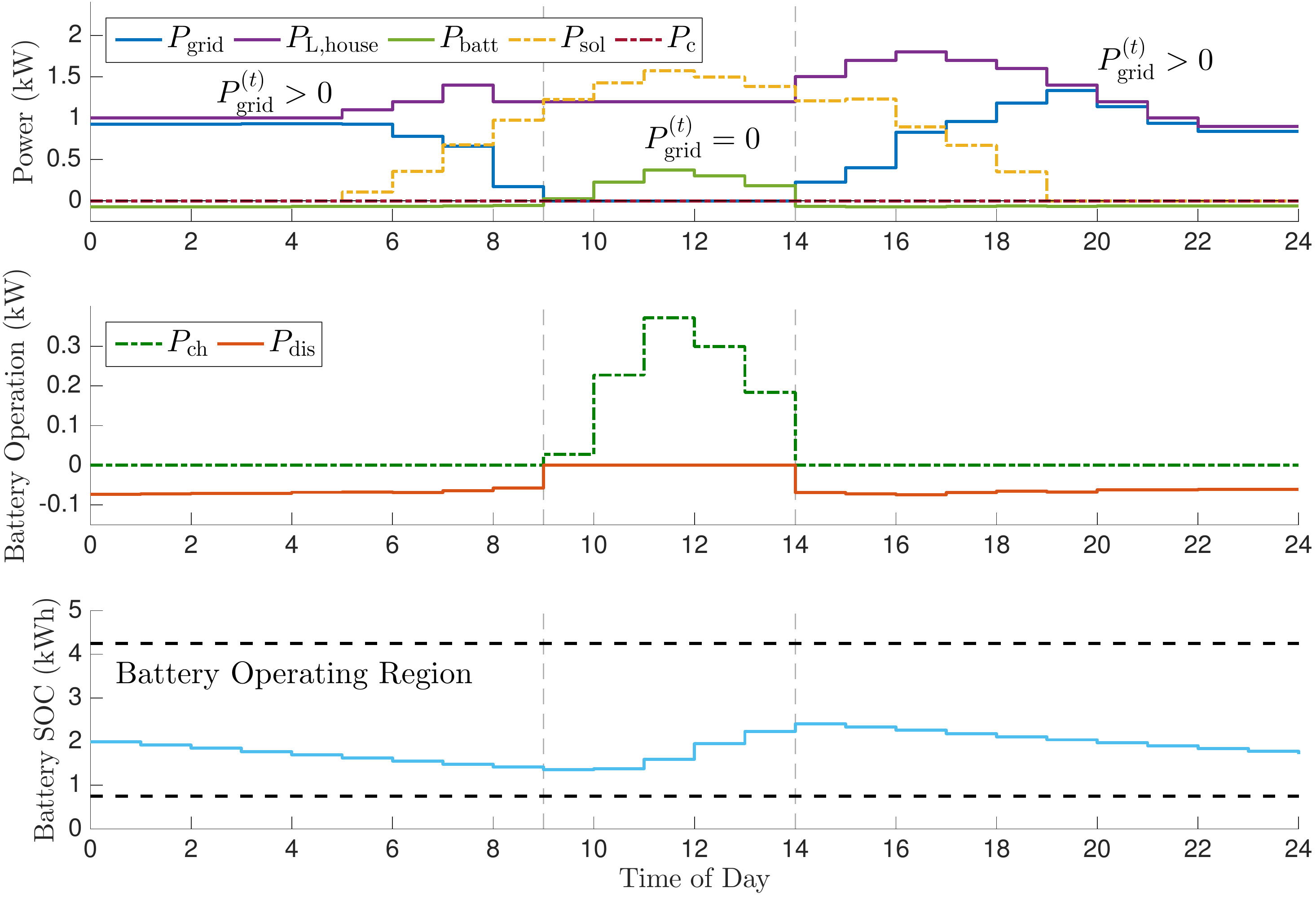} \vspace{-0.35cm}
		\caption{HEMS simulation results for a constant electricity price where excess power is not exported to the grid. Power profiles with HEMS optimization algorithm (top). ESS charging and discharging behavior (middle). ESS state of charge (bottom).}
		\label{fig:flat_ce} \vspace{-0.6cm}
	\end{figure}

	Next, we repeat the simulation above with the available solar \textit{increased} by 150\% and with the same objective function in~\eqref{eqn:fcost_tou} where the price of electricity is constant $c_e^{(t)}=\$0.11$/kWh $\forall t$, there is no ESS usage penalty $\alpha=\beta=0$, and excess power is not able to be exported to the grid. The simulation results for this case are shown in Fig.~\ref{fig:flat_ce_XLsolar_noChfcost}, demonstrating how simultaneous charging and discharging \textit{can} occur when there is a high penetration of solar power and no ESS usage penalty included in the cost function. Recall from the proof of Prop. \ref{thm:ce_tou_sub}, high penetration of solar power is when the available solar power $\Psol^{(t)}$ exceeds the load $P_{\rm L,house}^{(t)}$ in a situation where power cannot be exported to the grid. The ESS operating behavior in Fig.~\ref{fig:flat_ce_XLsolar_noChfcost} (middle) shows simultaneous ESS charging and discharging in the linear model under these operating conditions. Furthermore, we provide an additional simulation shown in Fig.~\ref{fig:flat_ce_XLsolar_wChfcost} with the same high penetration of solar as in the previous example but the cost function in~\eqref{eqn:fcost_tou} is modified such that the ESS charging and discharging penalty is included $\alpha=0.001$, $\beta=0$ (with the price of electricity still constant at $c_e^{(t)}=\$0.11$/kWh $\forall t$). The simulation results in Fig.~\ref{fig:flat_ce_XLsolar_wChfcost} (middle) with non-simultaneous ESS charging and discharging demonstrate how including a penalty on ESS usage in the cost function ensures that a solution with simultaneous ESS charging and discharging is suboptimal.
	\begin{figure}[t!]
		\centering
		\includegraphics[width=3.5in]{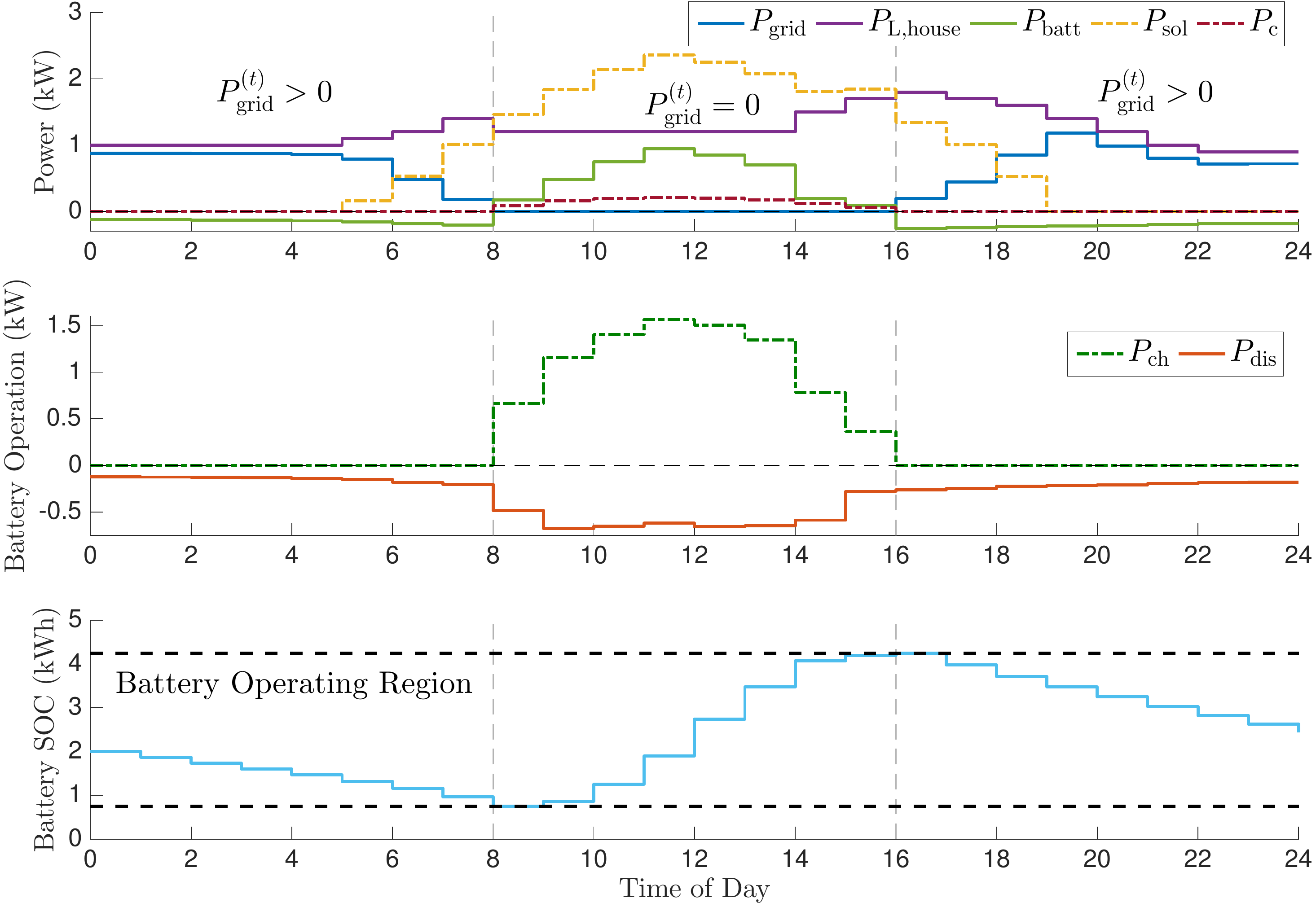} \vspace{-0.35cm}
		\caption{HEMS simulation results for a constant electricity price where excess power is not exported to the grid with a high penetration of solar and no ESS usage penalty in the optimization cost function. Power profiles with HEMS optimization algorithm (top). Simultaneous ESS charging and discharging behavior (middle). ESS state of charge (bottom).}
		\label{fig:flat_ce_XLsolar_noChfcost} \vspace{-0.25cm}
	\end{figure}
	\begin{figure}[t!]
		\centering
		\includegraphics[width=3.5in]{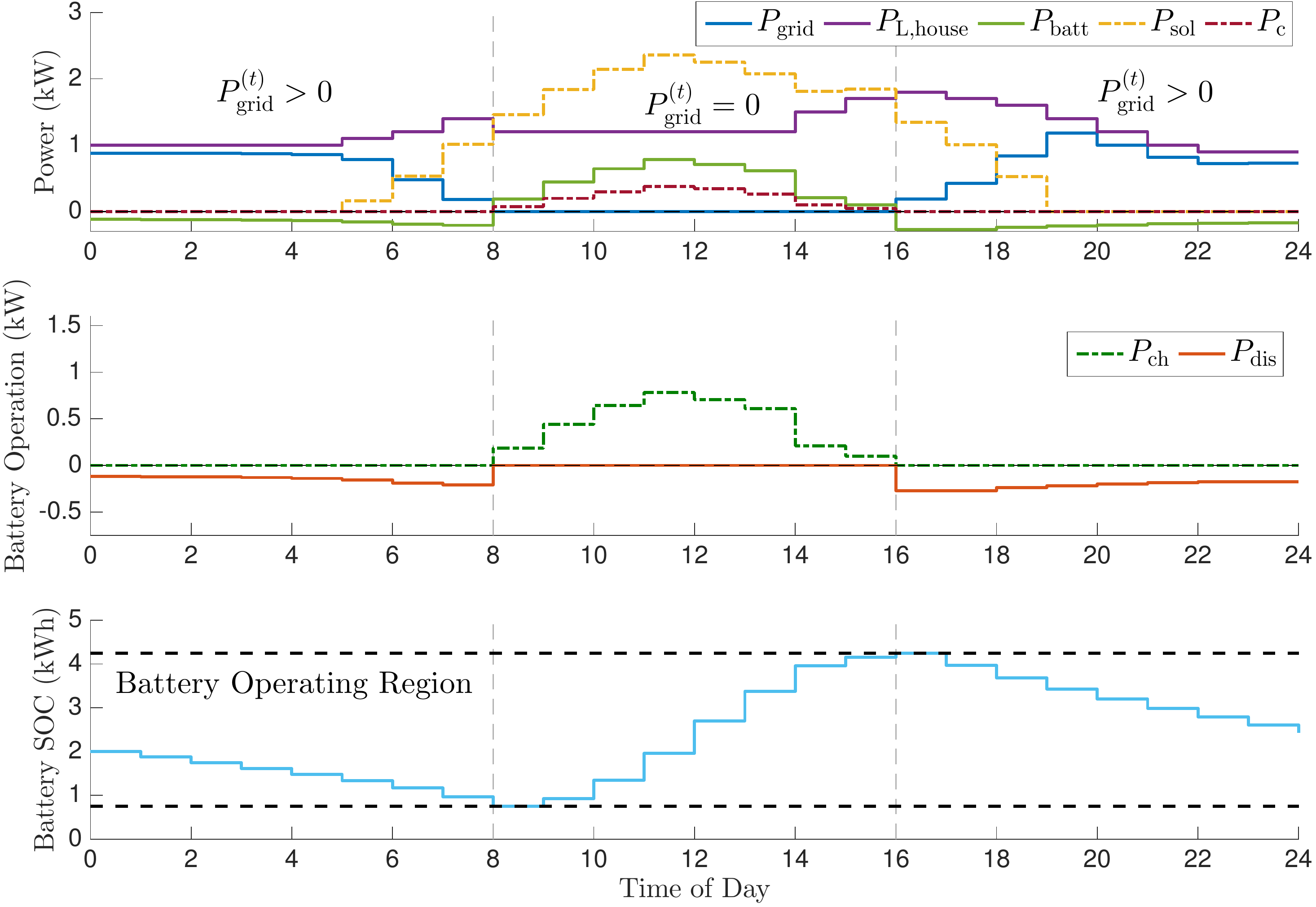} \vspace{-0.35cm}
		\caption{HEMS simulation results for a constant electricity price where excess power is not exported to the grid with a high penetration of solar and ESS usage penalty included in the optimization cost function. Power profiles with HEMS optimization algorithm (top). Non-simultaneous ESS charging and discharging behavior (middle). ESS state of charge (bottom).}
		\label{fig:flat_ce_XLsolar_wChfcost} \vspace{-0.25cm}
	\end{figure}

	Next, we show simulation results for the HEMS algorithm with the objective function in~\eqref{eqn:fcost_tou}, where $\alpha=0.001$ and $\beta=0$, under a TOU pricing structure where excess power cannot be exported to the grid. The TOU pricing schedule is given in Table~\ref{tab:TOU}, and is incorporated into the cost function. The simulation results for this case are provided in Fig.~\ref{fig:tou}. From Fig.~\ref{fig:tou} (top), we see that when $P_{\rm grid}^{(t)}=0$, the ESS charges or discharges based on the TOU pricing schedule. The ESS charges when there is excess solar after satisfying the total residential load. {During the On-Peak pricing period, once the house load surpasses the available solar power, the ESS discharges instead of using grid power.} When $P_{\rm grid}^{(t)}>0$, the ESS charges or discharges depending on the TOU pricing period at that time. The ESS non-simultaneous charging and discharging behavior is shown in Fig.~\ref{fig:tou} (middle) and (bottom).
	\begin{table}[t!]
		\centering
		\caption{TOU Pricing Schedule: Case 1} \vspace{-0.2cm}
		\label{tab:TOU}
		\begin{tabular}{c|c|c}
			Time of Use                                                   & Pricing Period & Price of Electricity $c_e^{(t)}$ \\ \hhline{=|=|=}
			9PM - 9AM                                                     & Off-Peak       & \$0.08/kWh           \\ \hline
			\begin{tabular}[c]{@{}c@{}}9AM - 2PM\\ 6PM - 9PM\end{tabular} & Shoulder       & \$0.13/kWh           \\ \hline
			2PM - 6PM                                                     & On-Peak        & \$0.18/kWh          
		\end{tabular} \vspace{-0.3cm}
	\end{table}
	\begin{figure}[t!]
		\centering
		\includegraphics[width=3.5in]{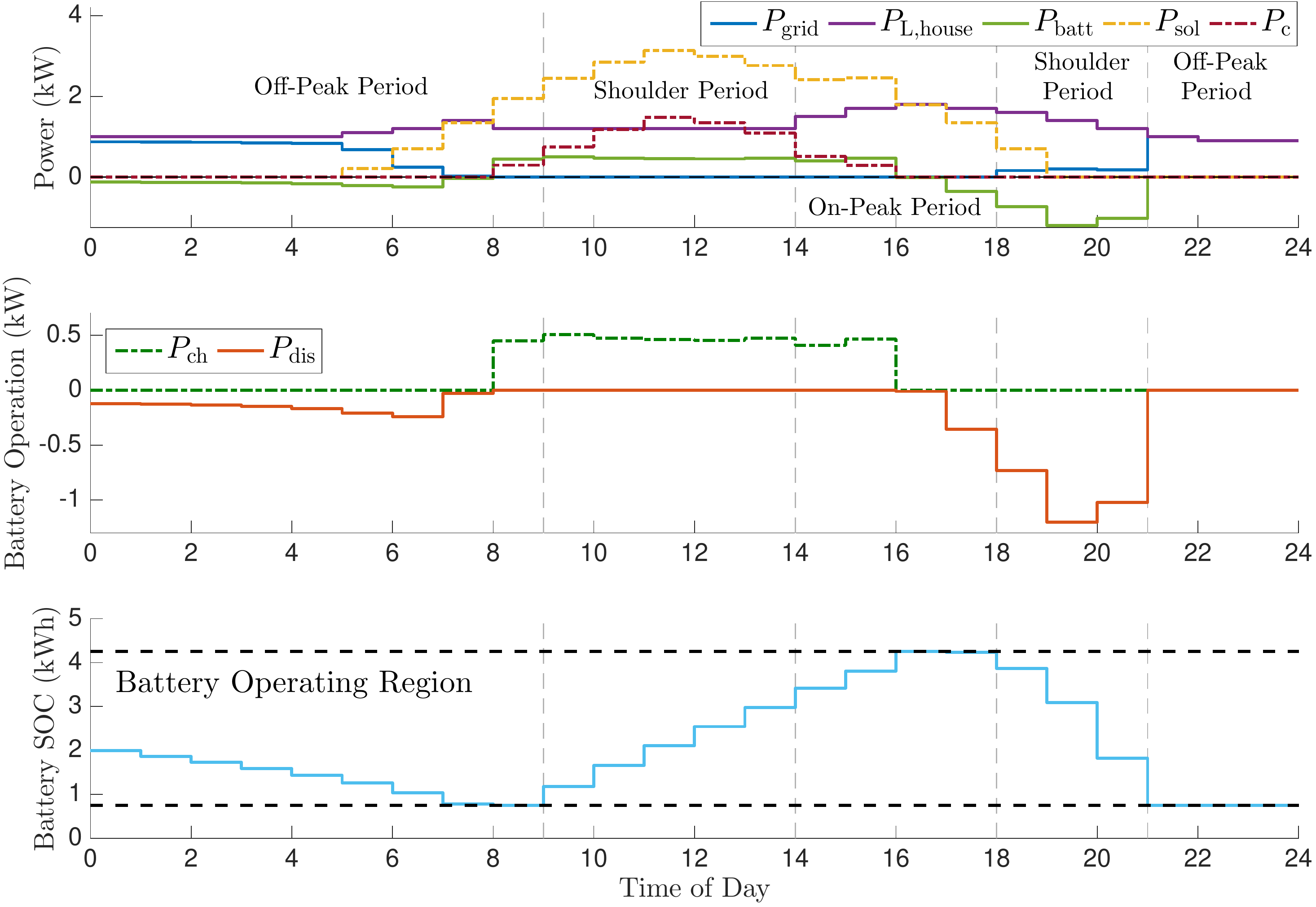} \vspace{-0.5cm}
		\caption{HEMS simulation results with TOU pricing and excess power when power export to the grid is not allowed. Power profiles with HEMS optimization algorithm (top). ESS charging and discharging behavior (middle). ESS state of charge (bottom).}
		\label{fig:tou} \vspace{-0.6cm}
	\end{figure}
	Lastly, we show simulation results for a HEMS algorithm under net metering with TOU pricing. The TOU pricing schedule is shown in Table~\ref{tab:TOU2}, and notice that $c_e^{(t)}=\$0$/kWh during the Off-Peak period. The TOU pricing schedule is incorporated into the cost function in~\eqref{eqn:fcost_tou}, and $\alpha=\beta=0.001$ to ensure non-simultaneous ESS charging and discharging during the Off-Peak period. The simulation for this case is shown in Fig.~\ref{fig:nm_tou}. From Fig.~\ref{fig:nm_tou} (top), we can see the excess solar is exported to the grid once the ESS is full. During the Off-Peak period, the ESS charges when power can be drawn from the grid for free. During the On-Peak period, the ESS discharges to contribute to the residential load and, due to the net metering structure, available solar is exported onto the grid. The ESS charging and discharging behavior is shown in Fig.~\ref{fig:nm_tou} (middle) and (bottom) to demonstrate non-simultaneous ESS charging and discharging in more detail. 
	\begin{table}[t!]
		\centering
		\caption{TOU Pricing Schedule: Case 2}
		\label{tab:TOU2} \vspace{-0.2cm}
		\begin{tabular}{c|c|c}
			Time of Use                                                   & Pricing Period & Price of Electricity $c_e^{(t)}$ \\ \hhline{=|=|=}
			9PM - 9AM                                                     & Off-Peak       & \$0.00/kWh           \\ \hline
			\begin{tabular}[c]{@{}c@{}}9AM - 2PM\\ 6PM - 9PM\end{tabular} & Shoulder       & \$0.13/kWh           \\ \hline
			2PM - 6PM                                                     & On-Peak        & \$0.18/kWh          
		\end{tabular} \vspace{-0.4cm}
	\end{table}
	\vspace{-0.5cm}
	\begin{figure}[t!]
		\centering
		\includegraphics[width=3.5in]{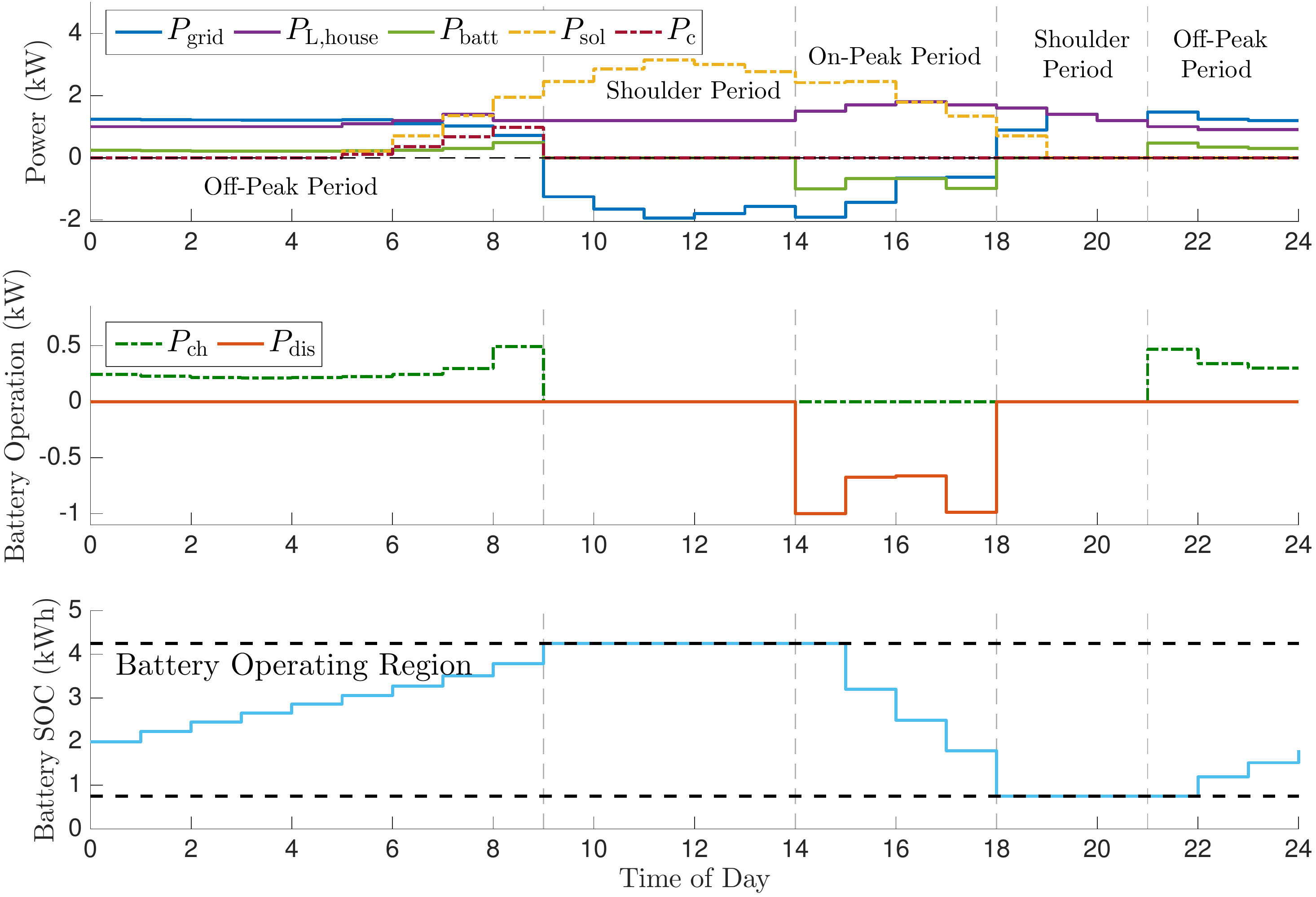} \vspace{-0.35cm}
		\caption{HEMS simulation results under net metering with TOU pricing. Power profiles with HEMS optimization algorithm (top). ESS charging and discharging behavior (middle). ESS state of charge (bottom).}
		\label{fig:nm_tou} \vspace{-0.65cm}
	\end{figure} 

	\section{Conclusion} \label{sec:conclusion}
	This work provides non-simultaneous charging and discharging guarantees for a linear ESS model in an MPC-based HEMS optimization algorithm for various cases. We used the KKT conditions, including both the dual and primal variables, to show that solutions to the optimization with simultaneous ESS charging and discharging were suboptimal in each situation studied in this work, meaning that nonconvex ESS models preventing simultaneous charging and discharging are unnecessary under these problem formulations. Additionally, we showed that including terms in the objective function $f_{\rm cost}(\mathbf{\underline{x}}_{N})$ to capture ESS lifetime considerations can improve the non-simultaneous charging and discharging guarantees for the ESS model in the situations studied in this work. If there are times when the cost of electricity $c_e^{(t)}=\$0$/kWh for some time $t$, including terms in the objective function $f_{\rm cost}(\mathbf{\underline{x}}_{N})$ to capture ESS lifetime considerations are necessary to ensure proper ESS dynamics with the proposed ESS model. Simulation results were provided to show the ESS charging and discharging behavior for each of the cases studied in this work.
	
	Future work includes ensuring non-simultaneous charging and discharging in stochastic ESS models or in electric vehicle (EV) models where the EV battery is used for flexible storage. Additionally, this work can be extended to other electricity pricing schemes such as feed-in tariffs. The KKT analysis we applied to a MPC-based HEMS optimization can be expanded to other convex ESS models and renewable energy research settings equipped with an ESS. \vspace{-0.3cm}
	
	%
	\ifCLASSOPTIONcaptionsoff
	\newpage
	\fi
	\bibliographystyle{IEEEtran}
	\bibliography{OptPowerSysBib}

	\begin{IEEEbiography}[{\includegraphics[width=1in,height=1.25in,clip,keepaspectratio]{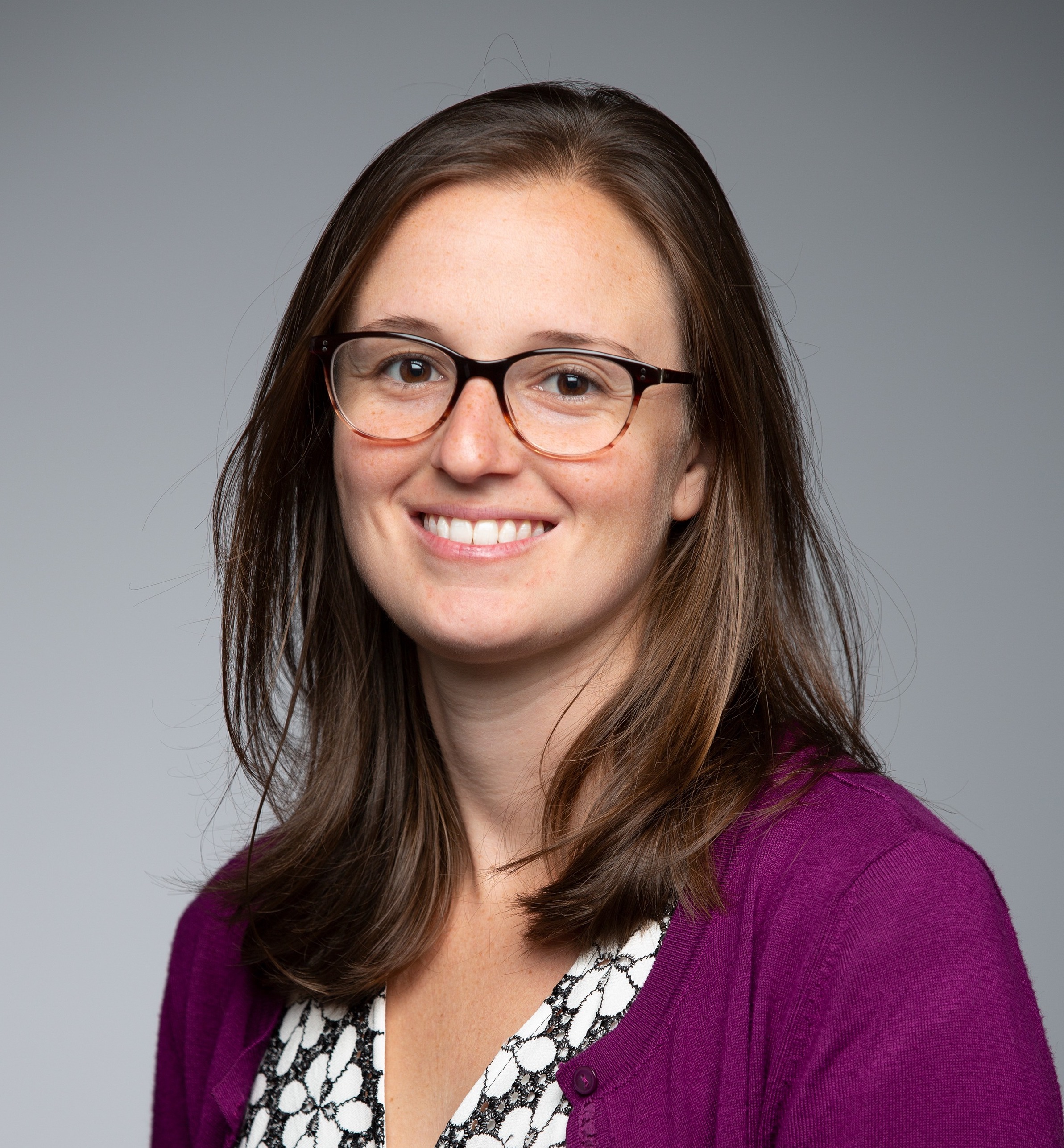}}]{Kaitlyn Garifi} is currently working towards her PhD in Electrical Engineering at University of Colorado Boulder. She received her M.S. in Electrical Engineering from University of Colorado Boulder in 2017, and both a B.S. in Bioengineering and B.A. in Mathematics from University of California, Santa Cruz, in 2015. Her research interests include stochastic optimization, renewable energy integration, and smart grid technologies.
	\end{IEEEbiography}
	\vspace{-1.2cm}
	\begin{IEEEbiography}[{\includegraphics[width=1in,height=1.25in,clip,keepaspectratio]{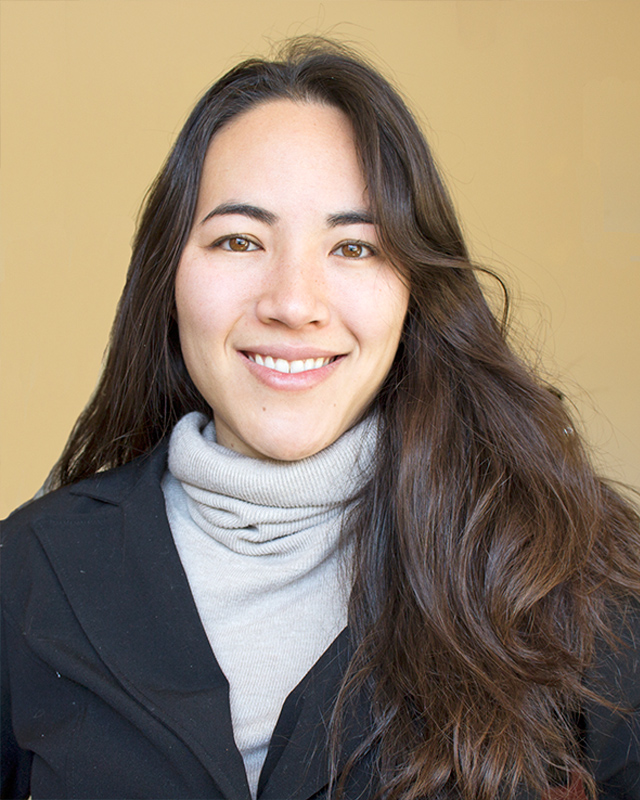}}]{Kyri Baker} received her PhD, M.S., and B.S. in Electrical and Computer Engineering from Carnegie Mellon University in 2014, 2010, and 2009, respectively. She is currently an Assistant Professor in the Civil, Environmental, and Architectural Engineering Department at the University of Colorado, Boulder. She also holds a courtesy appointment in the Department of Electrical, Computer, and Energy Engineering, and a joint appointment at the National Renewable Energy Laboratory (NREL) through the Renewable and Sustainable Energy Institute (RASEI). Prof. Baker's research focuses on optimization and control solutions for future smart grids, from the individual building level to the market level.
	\end{IEEEbiography}
	\vspace{-1.2cm}
	\begin{IEEEbiography}[{\includegraphics[width=1in,height=1.25in,clip,keepaspectratio]{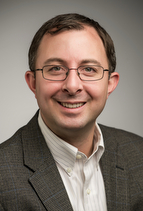}}]{Dane Christensen} is a Senior Engineer in the Buildings and Thermal Systems Engineering Center at the National Renewable Energy Laboratory (NREL) in Golden, CO. His team focuses on smart buildings, home-to-grid integration, and how emerging sensors, analytics and controls can enable cost-effective outcomes for building owners, utilities, manufacturers and society. Research is performed in the Energy Systems Integration Facility on NREL's campus, in collaboration with industry and academia. Dr. Christensen has over 40 technical publications, three issued patents, and numerous provisional patents and software copyrights. Dane received a PhD in Mechanical Engineering from University of California, Berkeley, in 2005 and a B.S. in Mechanical Engineering from Rice University in 2001.	
	\end{IEEEbiography}
	\vspace{-1.2cm}
	\begin{IEEEbiography}[{\includegraphics[width=1in,height=1.25in,clip,keepaspectratio]{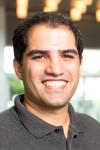}}]{Behrouz Touri} is an Assistant Professor of Electrical and Computer Engineering at the University of California San Diego. Prior to this, he was an Assistant Professor of Electrical, Computer, and Energy Engineering Department at the University of Colorado Boulder. Touri received a Ph.D. from the University of Illinois at Urbana Champaign in 2011. He is a recipient of an AFOSR Young Investigator Award 2016 and the 2018 Donald P. Eckman Award. His research interests include distributed computation and optimization, control over random networks, and game theory and multiagent systems. 
	\end{IEEEbiography}
	
	\vfill
	

\end{document}